 \documentclass[a4paper,10pt]{amsart}

\usepackage[english]{babel}
\usepackage{dsfont} 
\usepackage{graphicx}
\usepackage{color}
\usepackage{cleveref}

\allowdisplaybreaks

\usepackage{amssymb} 
\usepackage{amsthm} 

\newtheorem{thm}{Theorem}
\newtheorem{dfn}[thm]{Definition}
\newtheorem{lem}[thm]{Lemma}

\newtheorem{exa}[thm]{Example}
\newtheorem{prop}[thm]{Proposition}
\newtheorem{rem}[thm]{Remark}
\newtheorem{cor}[thm]{Corollary}

\newcommand{\nset}{\mathds{N}}

\newcommand{\rset}{\mathds{R}}

\newcommand{\pset}{\mathds{P}}

\newcommand{\codim}{\mathrm{codim}\,}

\newcommand{\lin}{\mathrm{lin}\,}
\newcommand{\Pos}{\mathrm{Pos}}
\newcommand{\range}{\mathrm{range}\,}
\newcommand{\rank}{\mathrm{rank}\,}

\newcommand{\supp}{\mathrm{supp}\,}

\newcommand{\folgt}{\;\Rightarrow\;}
\newcommand{\gdw}{\;\Leftrightarrow\;}

\newcommand{\cA}{\mathcal{A}}
\newcommand{\cB}{\mathcal{B}}
\newcommand{\cat}{\mathcal{C}}
\newcommand{\cK}{\mathcal{K}}
\newcommand{\cN}{\mathcal{N}}
\newcommand{\cM}{\mathcal{M}}
\newcommand{\cS}{\mathcal{S}}
\newcommand{\cX}{\mathcal{X}}
\newcommand{\cU}{\mathcal{U}}
\newcommand{\cZ}{\mathcal{Z}}
\newcommand{\Lin}{\mathrm{Lin}}

\newcommand{\sA}{\mathsf{A}}
\newcommand{\sB}{\mathsf{B}}

\author{Philipp J.\ di~Dio}
\address{Universit\"at Leipzig, Mathematisches Institut, Augustusplatz 10/11, D-04109 Leipzig, Germany}
\address{Max Planck Institute for Mathematics in the Sciences, Inselstra{\ss}e 22, D-04103 Leipzig, Germany}
\email{didio@uni-leipzig.de}

\author{Konrad Schm\"udgen}
\address{Universit\"at Leipzig, Mathematisches Institut, Augustusplatz 10/11, D-04109 Leipzig, Germany}
\email{schmuedgen@math.uni-leipzig.de}

\begin{title}{The multidimensional truncated {M}oment {P}roblem: {C}arath\'eodory {N}umbers}
\end{title}
\date{}

\begin{document}

\maketitle

\begin{abstract}
Let $\cA$ be a finite-dimensional subspace of $C(\cX;\rset)$, where $\cX$ is a locally compact Hausdorff space, and $\sA=\{f_1,\dots,f_m\}$ a basis of $\cA$. A sequence $s=(s_j)_{j=1}^m$ is called a moment sequence if $s_j=\int f_j(x) \, d\mu(x)$, $j=1,\dots,m$, for some positive Radon measure $\mu$ on $\cX$. Each moment sequence $s$  has a finitely atomic representing measure $\mu$. The smallest possible number of atoms is called the Carath\'eodory number $\cat_\sA(s)$. The largest number $\cat_\sA(s)$ among all moment sequences $s$ is the Carath\'eodory number $\cat_\sA$. In this  paper  the Carath\'eodory numbers $\cat_\sA(s)$ and $\cat_\sA$ are studied. In the case of differentiable functions methods from differential geometry are used. The main emphasis is on real polynomials. For a large class of spaces of  polynomials in one variable the number $\cat_\sA$ is  determined. In the multivariate case we obtain some lower bounds and we use results on zeros of positive polynomials to derive  upper bounds for the Carath\'eodory numbers.
\end{abstract}


\textbf{Key  words:} truncated moment problem, Carath\'eodory number, convex cone, positive polynomials

\section{Introduction}

The present paper continues the study of the truncated moment problem began in our previous papers \cite{schmud15} and \cite{didio17w+v+}. Here we investigate the Carath\'eodory number of moment sequences and moment cones.

Throughout this paper, we assume  that $\cX$ is a  locally compact topological Hausdorff space, $\cA$ is a \textbf{finite-dimensional} real linear subspace of $C(\cX; \rset)$ and  ${\sA}=\{f_1,\dots,f_m\}$  is a fixed basis of the vector space $\cA$.

Let $s=(s_j)_{j=1}^m$ be a real sequence and let $L_s$ be the linear functional on $\cA$ defined by $L_s(f_j)=s_j, j=1,\dots,m$. We say that $s$ is a {\it  moment sequence},  equivalently, $L_s$ is a {\it moment functional} on $\cA$, if there exists a (positive) Radon  measure $\mu$ on $\cX$  such that $f_j$ is $\mu$-integrable and
\[s_j=\int f_j(x) \, d\mu(x)\quad {\rm for}\quad j=1,\dots,m,\]
equivalently,
\[L_s(f)=\int_\cX f(x)~ d\mu(x)\quad {\rm for}\quad f\in \cA.\]
Such a measure $\mu$ is called a  representing measure of $s$ resp.\ $L_s$. The Richter--Tchakaloff Theorem (see Proposition \ref{richtercor} below) implies that each moment sequence has a $k$-atomic representing measure, where $k\leq m=\dim \cA$. The smallest number $k$ is called the \textit{Carath\'eodory number} $\cat_{\sA}(s)$ and the smallest number $K$ such that each  moment sequence $s$ has a $k$-atomic representing measure with $k\leq K$ is the {\it Carath\'eodory number} $\cat_{\sA}$. 

Let $L_s$ be a  moment functional. Determining  a $k$-atomic representing measure $\nu$ for $L_s$ is closely related  to the problem of finding quadrature or cubature formulas in numerical integration, see for instance \cite{davis84}, \cite{sobolev}. The Carath\'eodory number $\cat_{\sA}(s)$ corresponds  then  to the smallest possible number of nodes.

A large part of our considerations is developed in this general setup. Nevertheless we are mainly interested in the case when $\cA$ consists of real polynomials and $\cX$ is a closed subset of  $\rset^n$ or of the projective real space $\pset(\rset^n)$. In this case moment sequences are usually called \textit{truncated moment sequences} in the literature.

This paper is organized as follows.
In Section \ref{contfunction}, we define and investigate Carath\'eodory numbers and the  cone $\cS_\sA$ of moment sequences in the case when $\sA\subseteq C(\cX,\rset)$.  In Section \ref{difffunction}, we assume that  the functions of $\sA$ are differentiable and apply differential geometric methods to study the moment cone and  Carath\'eodory numbers.  Important technical tools are the total derivative $DS_{k,\sA}(C,X)$ associated with a $k$-atomic measure $\mu=\sum_{i=1}^k c_i\delta_{x_i}$ and  the  smallest number $\cN_{\sA}$  of atoms such that $DS_{k,\sA}(C,X)$ has full rank $m=|\sA|$. This number $\cN_{\sA}$ is a lower bound of the   Carath\'eodory number $\cat_{\sA}.$

The remaining four sections are concerned with polynomials. Section \ref{monoone} deals with  polynomials in one variable. For $\sA=\{1,x,\dots,x^m\}$ it is a classical fact  
that $\cat_\sA= \left\lceil \frac{m}{2}\right\rceil.$ We investigate a set  $\sA$  and its  homogenization   $\sB$ with gaps, that is,
\[\sA=\{1,x^{d_2},...,x^{d_m}\}\quad \text{and}\quad {\sB}=\{y^{2d},x^{d_2}y^{2d-d_2},...,x^{d_{m-1}}y^{2d-d_{m-1}},x^{2d}\},\]
where $ 0= d_1 <...<d_m=2d$. Our main result (Theorem \ref{thm:onedimCara})
gives sufficient conditions for the validity of the formula  $\cat_\sA=\cat_{\sB}=\left\lceil \frac{m}{2}\right\rceil$.

Sections \ref{monomulti}--\ref{realwaringrank} are devoted to  the multivariate case. Except from a  few simple cases the Carath\'eodory number $\cat_\sA$ is unknown for polynomials in several variables.  In Section \ref{monomulti} we give a new lower bound of $\cat_\sA$ and relate the number $\cN_\sA$ to the Alexander--Hirschowitz Theorem. Another group of main results of this paper is obtained  in Section \ref{carzeropos}. Here we use known results on zeros of non-negative polynomials to derive   upper bounds for  Carath\'eodory numbers  (Theorems \ref{thm:caraBounds}, \ref {thm:cara44}, and \ref{catzeroschoi}).  Section \ref{realwaringrank} deals with  signed Carath\'eodory numbers and the real Waring rank.

The multidimensional  truncated moment problem was first studied in the  Thesis of J.\ Matzke \cite{matzkePhD} and  independently by R.\ Curto and L.\ Fialkow \cite{curto3}, \cite{curto2}. It is an active research topic, see e.g.  \cite{richte57}, \cite{kemper68}, \cite{reznick92}, \cite{schmud15},  \cite{lauren09}, \cite{fialkow10}, \cite{curto13}, \cite{fialkoCoreVari}, \cite{fialkoMomProbSurv}, \cite{didio17w+v+}. 
Carath\'eodory numbers of multivariate polynomials have been  investigated by C.\ Riener and M.\ Schweighofer \cite{rienerOptima}. Carath\'eodory numbers of general convex  cones are studied in \cite{tuncel}.

For $r\in\rset$ let $\lceil r\rceil$ denote the smallest integer larger or equal to $r$.

\section{Carath\'eodory Numbers: Continuous Functions}\label{contfunction}

Let $\delta_x$ be the delta measure at $x\in \rset^n$, that is,  $\delta_x(M)=1$ if $x\in M$ and $\delta_x(M)=0$ if $x\notin M$. By a {\it signed $k$-atomic measure} $\mu$ we mean a signed measure $\mu=\sum_{j=1}^k c_j\delta_{x_j}$, where $x_1,\dots, x_k$ are pairwise different points of $\rset^n$ and $c_1,\dots,c_k$ are nonzero real numbers. If all numbers $c_1,\dots,c_k$ are positive, then $\mu$ is a positive measure and is called simply a {\it $k$-atomic measure}. The points $x_j$ are called the atoms of $\mu$. The zero measure is considered as $0$-atomic measure. 

%
%


A crucial result for our considerations  is  the \textit{Richter--Tchakaloff Theorem}  proved in \cite{richte57}. In the present context it can be stated as follows.

\begin{prop}\label{richtercor}
Each truncated moment sequence $s$ of $\sA$ has a $k$-atomic representing measure with $k\leq m=|\sA|$.
\end{prop}

%

%
%
%

\begin{dfn}
The moment cone $\cS_\sA \equiv {\cS}({\sA},\cX)$ is the set of all truncated $\cX$-moment sequences.
\end{dfn}

Obviously, $\cS_\sA$ is a convex cone in $\rset^m$. Since the functions $f_1,\dots,f_m$ form a vector space basis of $\cA$, 
it follows easily that $\rset^m=\cS_\sA - \cS_\sA$.

\begin{dfn}
The \emph{Carath\'eodory number} $\cat_{\sA}(s)\equiv \cat_{{\sA},\cX}(s)$ of $s\in {\cS}({\sA},\cX)$ is the smallest  $k$ such that $s$ has a $k$-atomic representing measure with all atoms in $\cX$. The \emph{Carath\'eodory number $\cat_{\sA} \equiv\cat_{\sA,\cX}$} of the moment cone ${\cS}({\sA},\cX)$ is the smallest number $\cat_{\sA}$ such that each moment sequence $s\in  {\cS}({\sA},\cX)$ has a $k$-atomic representing measure with all atoms in $\cX$ and $k\leq \cat_{\sA}$.
\end{dfn}

\begin{dfn}\label{def:signedCara}
The \emph{signed Carath\'eodory number} $\cat_{\sA,\pm}(s)\equiv
\cat_{\sA,\cX,\pm}(s)$ of  $s\in \rset^m$ is the smallest number $k$ such that $s$ has a signed $k$-atomic representing measure with all atoms in $\cX$. The \emph{signed Carath\'eodory number} $\cat_{\sA,\pm} \equiv \cat_{{\sA},\cX,\pm}$ is the smallest number $\cat_{\sA,\pm}$ such that every sequence $s$ has a signed $k$-atomic representing measure with all atoms in $\cX$ and $k\leq \cat_{\sA,\pm}$.
\end{dfn}

Since $\rset^m = \cS_\sA - \cS_\sA$ as noted above, Proposition \ref{richtercor} implies each vector $s'\in \rset^m$ has a signed $k$-atomic representing measure, where $k\leq 2m$, and  we have 
\begin{align}\label{carestimatem}
\cat_\sA(s)\leq \cat_{\sA} \leq m \quad \text{for}~~ s\in {\cS_\sA}\quad\text{and}\quad \cat_\pm(s')\leq \cat_{\sA,\pm } \leq 2m\quad \text{for}\quad s'\in \rset^m.
\end{align}

\begin{rem}
The above definitions of moment sequences, moment cones and Cara\-th\'eodory numbers make sense for Borel functions rather than continuous functions. For instance, let $x_1,\dots,x_m$ be pairwise different points of $\rset^n$ and let $\sA$ be the set of characteristic functions of the points $x_j$. Then it is easily verified that the Carath\'eodory number $\cat_{\sA}$ is equal to $m= |\sA|$.
\end{rem}

\begin{dfn}\label{def:sASA}
The \emph{moment curve} of $\sA$ in $\rset^m$ is defined by
\begin{equation}
s_{\sA}: \cX\rightarrow\rset^{m}, x\mapsto s_{\sA}(x) :=
\begin{pmatrix} f_1(x)\\ \vdots\\ f_m(x)\end{pmatrix}
\end{equation}
and we set
\begin{equation}
S_{k,\sA}: (\rset_{\geq 0})^k\times\cX^k\rightarrow\rset^{m}, (C,X)\mapsto  S_{k,\sA}(C,X): = \sum_{i=1}^k c_i\cdot s_\sA(x_i),
\end{equation}
where $C=(c_1,...,c_k)$, $X=(x_1,...,x_k)$.
\end{dfn}

Clearly, $s_{\sA}(x)$ is the moment sequence of the delta measure $\delta_x$ and  $S_{k,{\sA}}(C,X)$ is the moment sequence with representing measure $\mu = \sum_{i=1}^k c_i \delta_{x_i}$:
\begin{equation}
S_{k,{\sA}}(C,X) = \sum_{i=1}^k c_i s_{\sA}(x_i) = \bigg(\int_{\cX} f_j(x)\, d\mu(x)\bigg)_{j=1}^m ~.
\end{equation}
By Proposition \ref{richtercor}, each moment sequence $s\in \cS_\sA$ is of the form  $S_{m,\sA}(C,X)$ for some $(C,X)\in (\rset_{\geq 0})^m\times\cX^m$. Further, let us introduce a convenient  notation:
\begin{equation}
{\Pos}(\cA,\cX)\equiv {\Pos}(\cX):=\{ f\in \cA: f(x)\geq 0\quad \text{for}~~ x\in \cX\}.
\end{equation}

The following proposition restates a known result (see e.g.\ Lemma 3 and Proposition 27(i) in \cite{didio17w+v+}).

\begin{prop}\label{prop:zerosSupport}
Suppose that $s\in  \cS_\sA$ is a boundary point of $ \cS_\sA$. Then  there exists $p\in {\Pos}(\cA,\cK)$, $p\neq 0, $ such that $L_s(p)=0$ and each representing measure of $s$  is supported on the set of zeros $\cZ(p)$ of $p$.
\end{prop}

The next proposition is a crucial technical ingredient of many proofs given below.
The following condition is used at several places of this paper:
\begin{equation}\label{cond+}
\textit{There  exists  $e\in {\cA}$ such that $e(x)\geq 0$ for $x\in \cX$. }
\end{equation}

\begin{prop}\label{propstoboundary}
Let $s\in \cS_\sA$ and $x\in \cX.$  Suppose  that condition (\ref{cond+}) is satisfied. Define 
\begin{equation}\label{supcx}
c_s(x):=\sup\, \{ c\in \rset: (s- c\cdot s_{\sA}(x))\in \cS_\sA \, \}.
\end{equation}
Then $c_s(x)\leq e(x)^{-1}L_s(e)$ and  $(s-c_s(x)s_{\sA}(x))\in \partial\cS_{\sA}.$ 

If $\cK$ is compact, then the supremum in (\ref{supcx}) is attained, the moment cone $\cS_\sA$ is closed in $\rset^m$, and  we have
\begin{equation}\label{caboundary}
\cat_{\sA}\leq \max\, \{ \cat_\sA(s) : s \in \partial\cS_{\sA}\, \} +1.
\end{equation}
\end{prop}
\begin{proof}
Let $c\in \rset$. If $ (s- c s_{\sA}(x))\in \cS_\sA$, then $L_s-cl_x$ is a moment functional on $\cA$ and therefore $(L_s-cl_x)(e)\geq 0$, so that $c\leq e(x)^{-1} L_s(e)$. Hence $c_s(x)\leq e(x)^{-1}L_s(e).$ The definition of $c_s(x)$ implies that $s-c_s(x)s_{\sA}(x)$ belongs to the  boundary of $\cS_\sA$.

Since $\cX$ is compact, it was shown in \cite{fialkow10} that the moment cone $\cS_\sA$ is closed in $\rset^m.$ We choose  a sequence $(c_n)_{n\in \nset}$ such that $s- c_n s_{\sA}(x) \in\cS_\sA$ for all $n$ and $\lim_n c_n=c_s(x)$.  Then $s- c_n s_{\sA}(x)\to  s- c_s(x) s_{\sA}(x)$. Since $\cS_\sA$ is closed, we have $(s-c_s(x)s_{\sA}(x))\in \cS_{\sA}$, that is, the supremum (\ref{supcx}) is attained.  

Note that $(s-c_s(x)s_{\sA}(x))\in \partial\cS_{\sA}\cap\cS_\sA$. Obviously, $\cat_\sA(s) \leq \cat_\sA(s-c_s(x)s_{\sA}(x)) +1$. This implies the inequality (\ref{caboundary}). 
\end{proof}

The following example shows that the number $c_s(x)$  is not equal to
\begin{equation}\label{seconddefc}
\overline{c}_s(x) := \sup\, \{ c\in \rset : (s- c\cdot s_{\sA}(x))\in \overline{\cS_\sA} \, \}.
\end{equation}
However, if $s\in\mathrm{int}\ \cS_\sA$, then $c_s(x)=\overline{c}_s(x)$ by 
 Proposition \ref{proptwo}(vi) below.

\begin{exa}
Set $\cX=[-1,\pi]$, 
\[f_1(x) := 1,\quad f_2(x) := \begin{cases} 0 & x\in [-1,0]\\ \sin x & x\in (0,\pi] \end{cases},\quad f_3(x) := \begin{cases} x+1 & x\in [-1,0]\\ \cos x & x\in (0,\pi]\end{cases},\]
and $g_i = f_i|_{[-1,\pi)}$ for $i=1,2,3$. Set $\sA = \{f_1,f_2,f_3\}$  and  $\sB = \{g_1,g_2,g_3\}.$ Then $\cS_\sA$ is closed, but $\cS_\sB$ is  not closed. In fact,  $\overline{\cS_\sB} = \cS_\sA$. Let $s = s_\sA(-1) = (1,0,0)^T$, 
Then  $s' = s - s_\sA(0)/2 = (1/2,0,-1/2)^T = s_\sA(\pi)/2\in \partial \cS_\sA = \partial \cS_\sB$, but $s_\sA(\pi)\not\in \cS_\sA$. Thus $c_s(0)=0$ and $\overline{c}_s(0) = 1/2.$
\end{exa}

Recall from \cite{schmud15} the \emph{maximal mass function} $\rho_L(x)$ of a moment functional $L$:
\begin{equation}\label{eq:MaxMassDef}
\rho_L(x) := \sup\{ \mu(\{x\}): \mu\ \text{is a representing measure of}\ L\}, ~~
 x\in \cX.
\end{equation}

\begin{prop}\label{proptwo}
Suppose that condition (\ref{cond+}) holds and retain the notation from Proposition \ref{propstoboundary}.
%
\begin{itemize}
\item[\em (i)] $s- c\cdot s_\sA(x)\not\in \cS_\sA$ for all $c> c_s(x)$.

\item[\em (ii)] If $s\in\mathrm{int}\,\cS_\sA$, then $s- c\cdot s_\sA(x)\in\mathrm{int}\,\cS_\sA$ for all $c<c_s(x)$.

\item[\em (iii)] The map $\mathrm{int}\ \cS_{\sA}\ni s \mapsto c_s(x)\in \rset$ is concave and continuous for all $x\in\cX$.

\item[\em (iv)] The map $\cX\ni x\mapsto c_s(x)\in \rset$ is continuous for all $s\in\mathrm{int}\,\cS_\sA$.

\item[\em (v)] $c_s(x) = \rho_{L_s}(x)$.
\item[\em (vi)] If $s\in\mathrm{int}\ \cS_\sA$, then $c_s(x) = \overline{c}_s(x)$.
\end{itemize}
\end{prop}
\begin{proof}
(i) is clear  from the definition (\ref{supcx}).

(ii): Since $s$ is an inner point, there exists $\varepsilon>0$ such that $B_\varepsilon(s)\subset \mathrm{int}\,\cS_\sA$. From the  convexity of $\cS_\sA$ it follows that
\[B_{\frac{c_s(x)-c}{c_s(x)}\varepsilon}(s- c\cdot s_\sA(x)) \subset\mathrm{int}\,\cS_\sA \qquad\forall c<c_s(x).\]

(iii): Let $s,t\in\cS_\sA$ and $\lambda\in (0,1)$. Choose $c,c'\in \rset$ such that $c<c_s(x)$ and $c'<c_t(x)$. Then $s-c s_\sA(x)$ and $ t-c' s_\sA(x)$ are in  $\cS_\sA$. Since  $\cS_\sA$ is convex, we have
\begin{align*}
\lambda [s - c s_\sA(x)] &+ (1-\lambda)[t - c' s_\sA(x)]\\
&= [\lambda s + (1-\lambda) t] - [\lambda c + (1-\lambda) c'] s_\sA(x) \in\cS_\sA,
\end{align*}
i.e., $\lambda c + (1-\lambda) c' \leq c_{\lambda s + (1-\lambda)t}(x)$. Taking the suprema over $c$ and $c'$ it follows that $\lambda c_s(x) + (1-\lambda) c_t(x) \leq c_{\lambda s + (1-\lambda)t}(x)$. Hence $s\mapsto c_s(x)$ is a concave function and therefore continuous on $\mathrm{int}\ \cS_\sA$ by \cite[Thm.\ 1.5.3]{schne14}.

(iv): Let $x\in\cX.$ Let $K$ be a compact neighborhood of $x$ and $(x_i)_{i\in I}$ a net in $K$ such that $\lim_{i\in I} x_i=x$. Since $K$ is compact, we have $e(y) \geq \delta >0$ and $\|s_\sA(y)\|\geq \delta$ for $y\in K$. Hence $ c_s(y)$ is bounded on $K$, say by $k$, by Proposition \ref{propstoboundary}.
Since $s_\sA(y)$ is continuous,  there exist $M>0$
such that $\|c_s(y)  s_\sA(y)\|\leq M$ on $K$.
Further, from (i) and (ii) it follows that $\partial\cS_\sA\cap (s+\rset\cdot s_\sA(y)) = \{s - c_s(y) s_\sA(y)\}$ for  $y\in K$.

Define $s_y' := s - c_s(y) s_\sA(y)$. Then $s_y'\in B_M(s)\cap\partial\cS_\sA$ for all $y\in K$. Since $\partial\cS_\sA$ is closed and $B_M(s)$ is compact,  $B_M(s)\cap\partial\cS_\sA$ is also compact. Therefore, $(s_{x_i}')_{i\in I} \subseteq B_M(s)\cap\partial\cS_\sA$ has an accumulation point, say $a$. Since $\partial\cS_\sA$ is closed, $a\in\partial\cS_\sA$. Since $c_s(x_i)$ is bounded by $k$ and $s_\sA$ is continuous,
\[|\langle v, s_{x_i}' - s\rangle| = |\langle v, -c_s(x_i) s_\sA(x_i)\rangle| \leq k\cdot |\langle v,s_\sA(x_i)\rangle| \rightarrow k\cdot |\langle v,s_\sA(x)\rangle| = 0\]
for all $v \perp s_\sA(x)$, i.e., $a-s\in [-k,k]\cdot s_\sA(x)$, so that  $a\in s + [-k,k]\cdot s_\sA(x)$. Then
\[a\in \partial\cS_\sA\cap (s+[-k,k]\cdot s_\sA(x)) \subseteq \partial\cS_\sA\cap (s+\rset\cdot s_\sA(x)) = \{s_x'\},\]
so $(s_{x_i}')_{i\in I}$ has a unique accumulation point $s_x'$. Thus $\lim_{i\in I} s_{x_i}'=s_x'$.  This proves that the map $y\mapsto s_y'$ is continuous at $x$. Therefore,
\[\|s - s_y'\| \cdot \|s_\sA(y)\|^{-1} = \|c_s(y) s_\sA(y)\|\cdot \|s_\sA(y)\|^{-1} = \|c_s(y)\| = c_s(y)\]
is continuous at $x$. Since $x\in\cX$ was arbitrary, $x\mapsto c_s(x)$ is continuous on $\cX$.

(v): Let $c\in \rset$ be such that\ $\tilde{s}:= s-c\cdot s_\sA(x)\in\cS_\sA$. Then $L_s = L_{\tilde{s}} + c\cdot \delta_x.$ Hence   there is a representing measure $\mu$ of $s$ such that $c\leq \mu(\{x\})\leq \rho_{L_s}(x)$. Taking the supremum over $c$ yields $c_s(x)\leq \rho_{L_s}(x)$. 

Assume  that $c_s(x) < \rho_{L_s}(x)$. By the definition of $\rho_{L_s}(x)$,  there exist a $c\in (c_s(x),\rho_{L_s}(x))$ and a representing measure $\mu$ of $s$ such that $\mu(\{x\})=c.$ Then  $\tilde{\mu}:= \mu - c\cdot \delta_x$ is a positive Radon measure representing $\tilde{s} = s - c\cdot s_\sA(x)$. But $\tilde{s}\not\in\cS_\sA$ by (i), a contradiction. This proves that $c_s(x) \not < \rho_{L_s}(x).$  Thus,  $c_s(x) = \rho_{L_s}(x)$. 

(vi): Since $s\in\mathrm{int}\ \cS_\sA$, it follows from  (i) and (ii)  that 
$$\partial\cS_\sA\cap (s+\rset\cdot s_\sA(x)) = \{s_x'=s - c_s(x) s_\sA(x)\}.$$
Both numbers  $s- c_s(x) s_\sA(x)$ and $ s - \overline{c}_s(x) s_\sA(x)$ belong to the set on left hand side set. Hence  they are equal and therefore   $c_s(x) = \overline{c}_s(x)$.
\end{proof}

From  Proposition \ref{proptwo}(iii) we easily derive   that the supremum in (\ref{eq:MaxMassDef}) is attained if $\cX$ is compact. This was proved in \cite[Prop.\ 6]{schmud15} by using the weak topology on the set of representing measures and the Portmanteau Theorem.

The following example shows that (iv) is false in general if $s\in\partial\cS_\sA$.

\begin{exa}
Let $\{x_1,...,x_{10}\}$ be the zero set of the Robinson polynomial, $\cA$  the homogeneous polynomials of degree 6 on $\pset(\rset^2)$, and $s := \sum_{i=1}^{10} s_\sA(x_i)$. By Theorem 18 and  Example 18 in \cite{didio17w+v+}, $s$  is determinate. Therefore,
\[\rho_s(x)=c_s(x) = \begin{cases} 1 & \text{for}\ x\in\{x_1,...,x_{10}\},\\ 0 & \text{else.} \end{cases}\]
\end{exa}

If $K$ is not compact, then the supremum in (\ref{supcx}) is not attained in general. This is shown by the following simple example.

\begin{exa}
Let $\cX=\rset$, $\sA=\{1,x,x^2\}.$ Set $s = (1,0,1)^T = \frac{1}{2}(s_\sA(-1) + s_\sA(1))$. Clearly, $s_\sA(0)=(1,0,0)^T$. Then $c_s(0) = 1$, but   $s' = s - c_s(0) s_\sA(0) = (0,0,1)^T$ is not in $\cS_\sA$.
\end{exa}

The following theorem improves the first equality in (\ref{carestimatem}) and Proposition \ref{richtercor}.

\begin{thm}\label{worstupperbound}
Suppose that condition (\ref{cond+}) holds. If $m\geq 2$ and $\cX$ has at most $m-1$ path-connected components, then $\cat_{\sA}\leq m-1$.
\end{thm}
\begin{proof}
Obviously, the Carath\'eodory number $\cat_{\sA}$ depends only on the linear span $\sA$, but not on the particular basis $\sA$ of $\cA = \Lin\, \sA$. Hence we can assume without loss of generality that $e=f_m$. Since $e(x)>0$ on $\cX$ by assumption, $b_j:=f_j e^{-1}\in C(\cX)$ for $j=1,\dots,m$.  Set ${\sB}=\{b_1,\dots,b_m\}.$

Let $s$ be a moment sequence of $\sB$. First we prove that $s$ has a finitely atomic representing measure of a most $m-1$ atoms. Upon normalization we can assume that $s_m=1$. By Proposition \ref{richtercor}, $s$ has a $k$-atomic measure $\mu=\sum_{j=1}^k c_j\delta_{x_j}$, where $k\leq m$ and $x_j\in \cX$ and $c_j>0$ for all $j$. If $k<m$, we are done, so we can assume that $k=m$.  Since $\cX$ consists of at most $m-1$ path-connected components, it follows that at least two points $x_i$, say $x_1$ and $x_2$, are in  the same component, say $\cX_1$, of $\cX$. Then there is a connecting path $\gamma: [0,1]\rightarrow \cX_1$ such that $\gamma(0) = x_1$ and $\gamma(1) = x_2$. For $t\in [0,1]$ we denote  by $\Delta_t$ the simplex in $\rset^{m-1}\times\{1\}$ spanned by the points $s_{\sB}(x_1),s_{\sB}(\gamma(t)),s_{\sB}(x_3),\dots,s_{\sB}(x_{m}).$ Since $s_m=1$, we have $\sum_{j=1}^m c_j=1$. Hence $s:=(s_1,\dots,s_{m})$ belongs to the convex hull of $s_{\sB}(x_1),s_{\sB}(x_2),s_{\sB}(x_3),\dots,s_{\sB}(x_{m})$, that is, $s$ is in  the simplex $\Delta_1$. By decreasing $t$ to $0$ it follows from the continuity of $b_i$ that there exists a  $t_0\in [0,1]$ such that  $s$ belongs to the boundary of the simplex $\Delta_{t_0}$. Then $s$ is a convex combination of at most $m-1$ vertices. This yields a $k$-representing measure $\tilde{\mu}$ of $s$ with $k\leq m-1$.

Now we show that each moment sequence of $\sA$ has a $k$-atomic representing measure with  $k\leq m-1$. This in turn implies the assertion  $\cat_{\sA}\leq m-1$. Let $s^\prime$ be a moment sequence of $\sA$ and let $\mu^\prime$ be a finitely atomic representing measure of $s^\prime$. Let $s$ be the moment sequence of $\sB$ given by the measure $e(x)d\mu$. As shown in the preceding paragraph, $s$ has a $k$-atomic representing measure $\nu$, where  $k\leq m-1$. Then  $e(x)^{-1}d\nu$ is a $k$-atomic representing measure of $s^\prime$.
\end{proof}

\begin{cor}
Let $\cA=\{ p\in \rset[x_1,\dots,x_n]: \deg (p)\leq d\}$ and $\cX=\rset^n$. Then  
\[\cat_{\sA}\leq |\sA| - 1= \begin{pmatrix}n+d\\ n\end{pmatrix} - 1.\]
\end{cor}

We give two somewhat pathological  examples. Example \ref{exmp:worstcase2} shows that the assertion of Theorem \ref{worstupperbound} is not true  if the assumption on the function $e(x)$ is omitted.


\begin{exa}\label{exmp:worstcase2}
Set
\[\varphi(x) := \begin{cases} x & \text{for}\ x\in [0,1],\\
-x+2 & \text{for}\ x\in (1,2],\\
0 & \text{elsewhere.}\end{cases}\]
$\varphi_1(x):=\varphi(x)$, $\varphi_2(x):=\varphi(x-1)$, $\varphi_3(x):=\varphi(x-2)$. Then ${\sA}:=\{\varphi_1,\varphi_2,\varphi_3\}\subset C(\rset)$. Using  the moment sequence $s=(1,1,1)$ we find that $\cat_{\sA}=3$.
\end{exa}

Example \ref{exmp:spacefilling1} gives a three-dimensional moment cone with  $\cat_\sA = 1$. A slight modification of this idea yields for  $m\in \nset$ an $m$-dimensional space $\cA$  such that  $\cat_\sA = 1$.

\begin{exa}\label{exmp:spacefilling1}
Let $x_L$ and $y_L$ be the coordinate functions of a space filling curve \cite[Ch.\ 5]{saganSpaceFillingCurves}, i.e., $x_L,y_L:[0,1]\rightarrow [0,1]$ are continuous, nowhere differentiable on the Cantor set $\mathcal{C}$, differentiable on $[0,1]\setminus\mathcal{C}$, and the curve
\[(x_L,y_L):[0,1]\rightarrow [0,1]^2\]
is surjective. Set ${\sA}:=\{x_L,y_L,1\}$ and $\cX=[0,1]$. Then
\[s_\sA([0,1]) = [0,1]^2\times\{1\} \tag{$*$}\]
and the moment cone $\cS_\sA = \{(x,y,z) :  z\geq 0,\, 0\leq x\leq z,\, 0\leq y \leq z\}$ is full-dimensional. Clearly,  $(*)$ implies that  $\cat_\sA = 1$.
\end{exa}

\begin{rem}
In this paper the vector space $\cA$ is finite-dimensional. However the definitions of the moment cone and the Carath\'eodory number can be extended to infinite-dimensional vector spaces $\cA$. The following example shows that even in this case it is  possible that $\cat_\sA=1$.
Let $\sA=\{\varphi_n\}_{n\in\nset}$ be the coordinate functions of the $\aleph_0$-dimensional Sch\"onberg space filling curve \cite[Ch.\ 7]{saganSpaceFillingCurves}, i.e., $\varphi_n$ is continuous and nowhere differentiable on $[0,1]$ for all $n$, and set $\varphi_0=1$. Then
\[(\varphi_n)_{n\in\nset_0}: [0,1]\rightarrow \{1\}\times [0,1]^{\nset} \tag{$*$}\]
is surjective. The moment cone $\cS_\sA = \{(x_n)_{n\in\nset_0} : 0\leq x_n\leq x_0\}$ is full dimensional, closed, and $\cat_\sA = 1$ from $\mathrm{(*)}$.
\end{rem}

\begin{thm}\label{thm:LinIndepZeros}
Let $p\in \cA$ and $x_1,\dots,x_k\in \cX, k\in \nset$. Suppose  that $p(x)\geq 0$ for $x\in \cX$, $\cZ(p) = \{x_1,...,x_k\}$ and the set $\{s_\sA(x_i) : i=1,...,k\}$ is linearly independent. Then $\cat_\sA\geq k$.
\end{thm}
\begin{proof}
Let $s = \sum_{i=1}^k s_\sA(x_i)$. Clearly, $L_s(p)=0$ and hence $\supp\mu\subseteq\cZ(p)=\{x_1,\dots,x_k\}$ for any representing measure $\mu$ of $s$ by Proposition \ref{prop:zerosSupport}. Assume there is an at most $(k-1)$-atomic representing measure $\mu$. Without loss of generality we assume that $x_1\notin {\supp}\, \mu$, so $\mu$ is of the form $\mu = \sum_{i=2}^{k} c_i \delta_{x_i}$, $c_i\geq 0$. Then
\[0 = s - s = \sum_{i=1}^k s_\sA(x_i) - \sum_{i=2}^{k} c_i s_\sA(x_i) \quad\Rightarrow\quad
s_\sA(x_1) = \sum_{i=2}^k (c_i-1) s_\sA(x_i).\]
Since the set $\{s_\sA(x_i) : i=1,...,k\}$ is linear independent, this is a contradiction. Therefore, $k = \cat_\sA(s) \leq \cat_\sA$.
\end{proof}

Applications of the previous theorem will be given in Examples \ref{exmp:CaratheodoryOnBoundary} and \ref{exmp:HarrisPolynomial}. Deeper results on the connections between the Carath\'eodory number and the zeros of positive polynomials are treated in Section \ref{carzeropos}.

We derive some useful facts which will be used several times. We investigate some properties of the set 
\[\cS_k := \range S_{k,\sA}\]
of  moment sequences which are given by measures of at most $k$ atoms.

\begin{lem}\label{lem:convexSetmink}
For fixed $k\in \nset$ the following are equivalent:
\begin{itemize}
\item[\em (i)] $\cS_k$ is convex, or equivalently, $\cS_k+\cS_k\subseteq \cS_k$.
\item[\em (ii)] $\cS_k = \cS_{k+1}$.
\item[\em (iii)] $k\geq\cat_\sA$.
\end{itemize}
\end{lem}
\begin{proof}
(i)$\Rightarrow$(ii): Let $s = (1-\lambda) s_1 + \lambda s_\sA(x)\in\cS_{k+1}$ with $s_1,s_\sA(x)\in\cS_k$. Since $\cS_k$ is convex , $s\in\cS_k$. Hence $\cS_k=\cS_{k+1}$.

(ii)$\Rightarrow$(iii): Let $s = s_0 + \lambda_1 s_\sA(x_1) + ... + \lambda_l s_\sA(x_l)\in\cS_{k+l}$ be an arbitrary moment sequence. Set $s_i := s_0 + \lambda_1 s_\sA(x_1) + ... + \lambda_l s_\sA(x_i)$. Then
\begin{align*}
s_1 = s_0 + \lambda_1 s_\sA(x_1)\in\cS_{k+1}=\cS_k \;&\Rightarrow\; s_2 = s_1 + \lambda_2 s_\sA(x_2)\in\cS_{k+1}=\cS_k\\
&\;\;\vdots\\
&\Rightarrow s = s_{l-1} + \lambda_l s_\sA(x_l)\cS_{k+1}=\cS_k
\end{align*}
Thus $\cat_{\sA}\leq k$.

(iii)$\Rightarrow$(ii): Since $\cat_\sA\leq k$, we have $\cS_{\cat_\sA} \subseteq \cS_k \subseteq \cS_{\cat_\sA}$. Here the last inclusion follows from the mimimality of $\cat_\sA$. Hence, $\cS_k=\cS_{\cat_\sA}$ is convex.
\end{proof}

An immediate consequence of the preceding lemma are the following inclusions:
\begin{align}\label{stricticl}
\{0\}=\cS_0 \subsetneqq \cS_1\subsetneqq ... \subsetneqq \cS_{\cat_\sA} = \cS_{\cat_\sA+j},\quad j\in \nset.
\end{align}

\begin{prop}
\begin{itemize}
\item[\em (i)] $\cat_\sA = \min\{k: \cS_k\ \text{is convex}\} = \min\{k : \cS_k = \cS_{k+1}\}$.
\item[\em (ii)] For each $k=0,1,...,\cat_\sA$ there is a moment sequence $s$ such that $\cat_\sA(s)=k$.
\end{itemize}
\end{prop}
\begin{proof}
(i) follows at once from the minimality of $\cat_\sA$ in Lemma \ref{lem:convexSetmink}.

(ii): By (\ref{stricticl}),  we have $\cS_{k-1}\subsetneqq\cS_k$ for  $k=0,...,\cat_\sA$, where we set  $\cS_{-1}:=\emptyset$. Therefore, $\cS_{k} \setminus \cS_{k-1} \neq \emptyset$.
\end{proof}

\begin{prop}\label{limitcar} Suppose that condition (\ref{cond+}) is satisfied.
\begin{itemize}
\item[\em (i)] The cone $\cS$ is pointed, that is, $\cS\cap (-\cS)=\{0\}.$
\item[\em (ii)] If $\cS_1$ is closed, then $\cS_k$ is closed for all $k$.

\item[\em (iii)] If  the set $\cX$ is compact, then   $\cS_k$ is closed for all $k$.
\end{itemize}
\end{prop}
\begin{proof}
(i): Suppose that $s,-s\in \cS$. Using that $e(x)>0$ on $\cX$ we conclude that $L_s(e)\geq 0$ and $L_{-s}(e)=-L_s(e)\geq 0$, so  $L_s(e)=0$  and therefore  $s=0$.

(ii):  The proof follows by induction. Assume $\cS_1$ and $\cS_k$ is closed for some $k$. We show that  $\cS_{k+1}$ is also closed.

Let $(s_n)_{n\in \nset}$ be a sequence of $\cS_{k+1}$  such that $s_n\rightarrow s \in\overline{\cS_{k+1}}$. We can write $s_n = \alpha_n x_n + \beta_n y_n$ such that $x_n\in \cS_k, y_n\in \cS_1$, $\alpha_n,\beta_n\in [0,+\infty)$, and $\|x_n\|=\|y_n\|=1$ for all $n$. Since $\cS_k$ and $\cS_1$ are closed, the sets $\cS_k\cap B_1(0)$ and $\cS_1\cap B_1(0)$ are both  compact. 
Hence we can find a subsequence $(n_i)$ such that 
$x_{n_i}\rightarrow x\in\cS_k\cap B_1(0),$ and $y_{n_i}\rightarrow y\in\cS_1\cap B_1(0)$. Let us assume for a moment that the sequences $(\alpha_{n_i})$ and $(\beta_{n_i})$ are bounded. There is a subsequence $n_{i_j}$ such that  $\alpha_{n_{i_j}}\to \alpha \in[0,+\infty)$ and $\beta_{n_{i_j}}\to \beta \in[0,+\infty)$. Then  $s_{n_{i_j}} \rightarrow s = \alpha x + \beta y \in\cS_{k+1}$.  Thus, $\cS_{k+1}$ is closed.

We show that the sequence $(\beta_{n_i})$ is unbounded if $(\alpha_{n_i})$ is unbounded. Taking the standard scalar product $\langle\,\cdot\, ,\,\cdot\rangle$ in $\rset^m$,  we can uniquely write $y_n = y_n^\perp + y_n^\|$ with $x_n\,\|\,y_n^\|$,\, $x_n\perp y_n^\perp$. Then
\[ \|s_{n_i}\|^2 = \|\alpha_{n_i} x_{n_i} + \beta_{n_i} y_{n_i}^\| \|^2 = \|\alpha_{n_i} x_{n_i} + \beta_{n_i} y_{n_i}^\| \|^2 + \beta_{n_i}^2 \|y_{n_i}^\perp\|^2 \geq \beta_{n_i}^2 \|y_{n_i}^\perp\|^2.\]
Since $(s_{n_i})$ converges, the sequence $(\|s_{n_i}\|)$ is bounded by some $k$. Thus,
\[ k\geq \|\alpha_{n_i} x_{n_i} + \beta_{n_i} y_{n_i}^\| \| \geq | \alpha_{n_i} \|x_{n_i}\| - \beta_{n_i} \|y_{n_i}^\| \| | = | \alpha_{n_i} - \beta_{n_i} \|y_{n_i}^\| \| |\]
and if $(\alpha_{n_i})$ is unbounded, so $(\beta_{n_i})$ is unbounded. The same reasoning shows that  $(\alpha_{n_i})$ is unbounded if $(\beta_{n_i})$ is unbounded.

If the sequence $(\beta_{n_i})$ is unbounded, $(y_{n_i}^\perp )$ converges to $0$ and hence $y=-x$. Since $\cS$ is pointed by (i), this implies $x=y=0$, a contradiction to $\|x\|=\|y\|=1$. This completes the proof.

(iii): By (ii) it suffices to prove that $\cS_1$ is closed.  Clearly, condition (\ref{cond+}) implies that $s_\sA(x)\neq 0$ for all $x\in\cX$. Since $\sA\subseteq C(\cX,\rset)$ and  $\cX$ is compact, we have  $\|s_\sA\|^{-1} s_\sA\in C(\cX,S^{m-1})$ ($S^{m-1}$ denotes the unit sphere in $\rset^m$) and $\range \|s_\sA\|^{-1} s_\sA$ is closed. Hence, $\cS_1 \equiv \rset_{\geq 0}\cdot \range \|s_\sA\|^{-1} s_\sA$ is closed.
\end{proof}

More on the moment cone can be found in Proposition \ref{prop:moreOnMomentCone}.

%

\section{Caratheodory Numbers: Differentiable Functions}\label{difffunction}

In the rest of this paper we assume that $\cX=\rset^n$ or $\pset(\rset^n)$ and $\cA$ is a finite-dimensional linear subspace of $C^r(\rset^n;\rset)$, $r\in\nset$.

Clearly, $S_{k,\sA}$ in Definition \ref{def:sASA} is a $C^r$-map of $\rset_{\geq 0}^k\times \rset^{kn}$ into $\rset^m$. Let $DS_{k,\sA}$ denote its total derivative. We can write
\begin{equation}\begin{split}\label{eq:totalderivative}
DS_{k,\sA} &= (\partial_{c_1}S_{k,{\sA}},\partial_{x_1^{(1)}}S_{k,{\sA}},...,\partial_{x_1^{(n)}}S_{k,{\sA}},\partial_{c_2}S_{k,{\sA}},...,\partial_{x_k^{(n)}}S_{k,{\sA}})\\
&= (s_{\sA}(x_1), c_1\partial_1 s_{\sA}|_{x=x_1},...,c_1\partial_n s_{\sA}|_{x=x_1},s_{\sA}(x_2),...,c_k \partial_n s_{\sA}|_{x=x_k}).
\end{split}\end{equation}

The following number is crucial in what follows.

\begin{dfn}\label{NAdefinition}
\begin{equation}
\cN_{\sA} := \min\{k\in\nset : \rank DS_{k,{\sA}} = m\},
\end{equation}
i.e., $\cN_{\sA}$ is the smallest number $k$ of atoms such that $DS_{k,{\sA}}$ has full rank $m=|\sA|$.
\end{dfn}

A lower bound for $\cN_{\sA}$ is given by the following proposition.

\begin{prop}\label{NAlowerbound}
We have $\left\lceil \frac{|{\sA}|}{n+1}\right\rceil \leq \cN_{\sA}$. If all functions $f_i$ are homogeneous of the same degree, then  $\left\lceil \frac{|{\sA}|}{n}\right\rceil \leq \cN_{\sA}$.
\end{prop}
\begin{proof}
Since $DS_{k,{\sA}}$ has $|{\sA}|$ rows and each atom contributes $n+1$ columns, we need at least $k \geq \frac{|{\sA}|}{n+1}$ atoms for full rank. Thus,  $\cN_{\sA}\geq \left\lceil \frac{|{\sA}|}{n+1}\right\rceil.$

If all functions $f_i$ are homogeneous of degree $r$, then $f_i(\lambda x) = \lambda^r f_i(x)$ and so $\delta_{\lambda x} = \lambda^r \delta_x$. Hence $DS_{1,{\sA}}$ has rank at most $d$ and  kernel dimension at least $1$. Therefore, at least $k \geq \frac{|{\sA}|}{n}$ atoms are needed, so that $\cN_{\sA}\geq \left\lceil \frac{|{\sA}|}{n}\right\rceil.$
\end{proof}

\begin{exa}\label{exmp:worstcase}
Let $\varphi\in C^\infty_0(\rset),  \varphi\neq 0,$ and $\supp(\varphi)\subseteq (0,1)$. Set $\varphi_i(x):=\varphi(x-i+1)$ for $i=1,...,m\in\nset$ and ${\sA} := \{\varphi_1,...,\varphi_m\}$. Then $\partial s_{\sA}(x) = \varphi_i'(x) e_i$ for $x\in (i-1,i)$ or $0$ otherwise. Then $\cN_{\sA}= \cat_{\sA} = m.$
\end{exa}

\begin{thm}\label{thm:signedCaraUpperBound}
Suppose that $\sA\subseteq C^1(\rset^n,\rset)$. Then
%
\begin{equation}
\cat_{\sA,\pm}\leq 2\cN_\sA.
\end{equation}
Set $C=(1,...,1)\in\rset^{\cN_\sA}$.
There exists  $X\in\rset^{\cN_\sA\cdot n}$ and an open neighborhood $U$ of $(C,X)$ such that for every $\varepsilon >0$ there are  $(C_\varepsilon,X_\varepsilon) \in U$ and $\lambda_\varepsilon\in \rset$ such that
\begin{equation}\label{eq:pmCaraLinComb}
s = \lambda_\varepsilon (S_{\sA,\cN_\sA}(C,X) - S_{\sA,\cN_\sA}(C_\varepsilon,X_\varepsilon)).
\end{equation}
\end{thm}
\begin{proof}
It clearly suffices to prove the second part of the theorem. The first assertion follows then from the second.

Since $DS_{\cN_\sA}$ has full rank, there is a $(C,X)\in\rset_{> 0}^{\cN_\sA}\times\rset^{\cN_\sA n}$ such that $DS_{\cN_\sA}(C,X)$ has full rank. Since scaling the columns of $DS_{\cN_\sA}(C,X)$ does not change the rank, we can assume without loss of generality that $C=(1,...,1)$. Since the determinant is continuous there is an open neighborhood $U$ of $(C,X)$ such that
\[S_{\cN_\sA,\sA}(C,X) \in \text{int}\, S_{\cN_\sA,\sA}(U).\tag{$*$}\]
Let $s\in\rset^m$. By ($*$) there is a $(C_\varepsilon,X_\varepsilon)\in U$ such that $S_{\sA,\cN_\sA}(C,X) - S_{\sA,\cN_\sA}(C_\varepsilon,X_\varepsilon)$ is a multiple of $s$, i.e., (\ref{eq:pmCaraLinComb}) holds for some  $\lambda_\varepsilon\in\rset$.
\end{proof}

\begin{dfn}\label{defi:singularregular}
Let $n,k\in\nset$ with $k\geq\cN_\sA$. A $k$-atomic measure $(C,X)$ on $\rset^n$ is called \emph{regular (for $S_{k,\sA}$)} iff $DS_{k,\sA}(C,X)$ has full rank. Otherwise the measure $(C,X)$ is called \emph{singular (for $S_{k,\sA}$)}.

A real sequence $s=(s_\alpha)_{\alpha\in \sA}$ is called \emph{regular} iff $S_{k,\sA}^{-1}(s)$ is empty (that is, $s$ is not a moment sequence) or consists solely of regular measures. Otherwise, $s$ is called \emph{singular}.
\end{dfn}

\begin{thm}\label{densethm}
Suppose that ${\sA}\subset C^r(\rset^n;\rset)$ and  $r > \cN_{\sA}\cdot (n+1) - m$. Then \begin{equation}\label{eq:NAlessequalCA}
\cN_{\sA} \leq \cat_{\sA}.
\end{equation}
Further, the set of moment sequences $s$ which can be represented by less than $\cN_{\sA}$ atoms has $|{\sA}|$-dimensional Lebesgue measure zero in $\rset^m$.
\end{thm}
\begin{proof}
By Proposition \ref{NAlowerbound} we have $r > \cN_{\sA}\cdot (n+1) - m \geq 0$, so that  $r \geq 1$.

The moment sequences which can be represented by less than $\cN_{\sA}$ atoms are singular values. Hence  the second assertion follows from Sard's Theorem \cite{sard42}.

To prove (\ref{eq:NAlessequalCA}) assume to the contrary that $\cat_{\sA} < \cN_{\sA}$. Then every moment sequence in the moment cone is singular. This is a contradiction to Sard's Theorem since the moment cone has non-empty interior.
\end{proof}

\begin{rem}\label{rem:signedCarabounds}
Theorem \ref{densethm} also holds for the signed Carath\'eodory number with verbatim the same proof. With Theorem \ref{thm:signedCaraUpperBound} we get
\begin{equation}
\cN_\sA \leq \cat_{\sA,\pm} \leq 2\cN_\sA
\end{equation}
for ${\sA}\subset C^r(\rset^n;\rset)$ and  $r > \cN_{\sA}\cdot (n+1) - m$. Without these conditions the lower bound needs not to hold, neither for $\cat_\sA$ nor for $\cat_{\sA,\pm}$, see \cite[pp.\ 317--318]{federerGeomMeasTheo}.
\end{rem}

\begin{prop}
Suppose that $\sA\subseteq C^1(\rset^n,\rset)$ and $\{x\in\rset^n : s_\sA(x)=0\}$ is bounded. Let $\gamma\in C^1(\rset^n,\rset)$ be such that $\gamma(x)\geq 1$ and $\lim_{|x|\rightarrow\infty} \frac{f_i(x)}{\gamma(x)}=0$ for $i=1,\dots,m$ and let $s$ be a moment sequence of $\sA$. Set
\begin{equation}\label{eq:GammaSet}
\Gamma_{l,c}(s) := S_{\sA,\cat_\sA(s)+l}^{-1}(s) \cap \left\{ (C,X)\in\rset_{\geq 0}^{\cat_\sA(s)+l}\times \rset^{n\cdot(\cat_\sA(s)+l)} \,\middle|\, \sum_i c_i \gamma(x_i) \leq c \right\}
\end{equation}
with $l\in\nset_0$ and $c\geq 0$. Then:
\begin{enumerate}
\item[\textit{(i)}] $\Gamma_{l,c}(s)$ is closed for all $l\in\nset_0$ and $c\geq 0$.
\item[\textit{(ii)}] $\Gamma_{0,c}(s)$ is compact for all $c\geq 0$.
\end{enumerate}
If, in addition, $s$ is regular, then:
\begin{enumerate}
\item[\textit{(iii)}] $\exists c\geq 0: \Gamma_{l,c}(s)$ unbounded $\gdw$ $l\geq 1$.
\item[\textit{(iv)}] $\Gamma_{l,c}(s)$ compact $\forall c\geq 0$ $\gdw$ $l=0$.
\end{enumerate}
\end{prop}
\begin{proof}
(i): If $f\in C(\rset^n,\rset^m)$ and $K\subseteq\rset^m$ is closed, then $f^{-1}(K)$ is closed by the continuity of $f$. Since both intersecting sets in (\ref{eq:GammaSet}) are of the form $f^{-1}(K)$, they  are closed and so is their intersection.

(ii): Suppose  $\Gamma_{0,c}$ is non-empty. Since $\Gamma_{0,c}$ is closed by (i), it  suffices to show that it is bounded. Assume to the contrary that it is unbounded and let $(C^{(i)},X^{(i)})$ be a sequence such that $\lim_{i\to \infty} \, \|(C^{(i)},X^{(i)})\|=\infty$. Since
\[0 \leq c_j^{(i)} \leq c_j^{(i)} \gamma(x_j) \leq \sum_{l} c_{l}^{(i)} \gamma(x_{l}) \leq c\]
the sequence $(C^{(i)})$ is bounded. The sequence $(X^{(i)})$ is unbounded. After renumbering and passing to subsequences we can assume that $c_j^{(i)}\rightarrow c_j^*$ for all $j$, $x_j^{(i)}\rightarrow x_j^*$ for $j=1,...,k$ and $\|x_j^{(i)}\|\rightarrow\infty$ for $j=k+1,...,\cat_\sA(s)$ as $i\to \infty$. Since 
\begin{align*}
s = S_{\sA,\cat_\sA}((C^{(i)},X^{(i)})) &= \sum_{j=1}^k c_j^{(i)} s_\sA(x_j^{(i)}) + \sum_{i=k+1}^{\cat_\sA} c_j^{(i)} s_\sA(x_j^{(i)})
\end{align*}
for all $i$, it follows that 
\begin{align*}
s &= \lim_{i\rightarrow\infty} \sum_{j=1}^k c_j^{(i)} s_\sA(x_j^{(i)}) + \lim_{i\rightarrow\infty} \sum_{j=k+1}^{\cat_\sA} c_j^{(i)} s_\sA(x_j^{(i)})\\
&= \sum_{j=1}^k c_j^* s_\sA(x_j^*) + \lim_{i\rightarrow\infty} \underbrace{\sum_{j=k+1}^{\cat_\sA}  \underbrace{\frac{s_\sA(x_j^{(i)})}{\gamma(x_j^{(i)})}}_{\rightarrow 0} \cdot \underbrace{c_j^{(i)} \gamma(x_j^{(i)})}_{\leq c}}_{\rightarrow 0}\\
&= \sum_{j=1}^k c_j^* s_\sA(x_j^*).
\end{align*}
Therefore, $\mu^*=((c_1^*,...,c_k^*),(x_1^*,...,x_k^*))$ is a $k$-atomic representing measure of $s$ with $k\leq\cat_\sA(s)$. By the minimality of $\cat_\sA(s)$,    $k=\cat_\sA(s)$. Hence all sequence $(x_j^{(i)})$ are bounded which is a contradiction. Thus $\Gamma_{0,c}(s)$ is bounded.

It is clear that (iii) and (iv) are equivalent. Thus it suffices to prove (iii).

(iii) ``$\Rightarrow$'': By (ii),  if $\Gamma_{l,c}(s)$ is unbounded, we find a $k$-atomic representing measure with $k<\cat_\sA(s)+l$, i.e., $l\geq 1$.

(iii) ``$\Leftarrow$'': We will show that there is a $c>0$ such that for every $x\in\rset^n$ there is a representing measure $\mu$ in $\Gamma_{1,c}(s)$ which has  $x$ as an atom. This will prove that $\Gamma_{1,c}$, hence  also $\Gamma_{l,c}$, is unbounded for  $l\geq 1$.

Let $\mu_0 = (C_0,X_0) = ((c_{0,1},...,c_{0,\cat_\sA(s)}),(x_{0,1},...,x_{0,\cat_\sA(s)}))$ be a representing measure of $s$. Set
\[c := \int \gamma~d\mu_0 +1.\]
Since $s$ is regular, all representing measures have full rank. Hence there exist variables $y_1,...,y_m$ from $c_1,...,c_{\cat_\sA(s)},x_{1,1},...,x_{\cat_\sA(s),n}$ such that $D_y S(\mu_0)$ is a square matrix with full rank. Then
\[F((C,X),t) = S_{\cat_\sA(s),\sA}((C,X)) - s + t\cdot s_\sA(x)\]
is a $C^1$-function such that $F((C_0,X_0),0)=0$ and $D_y F((C_0,X_0)) = D_y S(\mu_0)$ is bijective.  Thus, $F$ fulfills all assumptions of the implicit function theorem, hence there are an $\varepsilon >0$ and a $C^1$--function $(C(t),X(t))$ such that $F((C(t),X(t)),t)=0$ for all $t\in (-\varepsilon, \varepsilon)$. Since $c_{i,0}>0$, there is $t_0\in (0,\varepsilon)$ such that $c_i(t_0)>0$ for all $i$, so 
\[\mu(t_0) = \sum_{i=1}^{\cat_\sA(s)} c_i(t_0) \delta_{x_i(t_0)} + t_0 \delta_{x} \quad\text{with}\quad \int\gamma~d\mu(t_0) \leq c\]
is a $(\cat_\sA(s)+1)$-atomic representing measure of $s$ which has $x$ as an atom.
\end{proof}

(iii) and (iv) no longer hold if $s$ is singular. E.g., let $s$ be moment sequence of the measure $\mu=\sum_{i=1}^{10} \delta_{x_i}$ where $x_i$ are the ten zeros of the Robinson polynomial, then $S_{k,\sA}^{-1}(s) \subseteq [0,10]^k\times \{x_1,...,x_{10}\}^k$ is compact for all $k\geq 10$.

The next proposition  summarizes a number of basic properties of the sets $\cS_k$ and the Carath\'eodory number. Recall that $B_\rho(t)$ is the ball with center $t$ and radius $\rho$ in $\rset^m$.

\begin{prop}\label{prop:moreOnMomentCone}
%
\begin{itemize}
\item[\em (i)] Suppose that $\cS_{k-1}$ is closed for some $k$,   $\cN_\sA\leq k \leq\cat_\sA$. Then there exist a moment sequence $s$ and an $\varepsilon>0$ such that $\cat_\sA(t)=k$ for all $t\in B_\varepsilon(s)$.
\item[\em (ii)] $s\in\mathrm{int}\,\cS_{\cat_\sA}$ if and only if 
there exists $(C,X)$ such that $S_\sA(C,X)=s$ and $\rank DS_\sA(C,X) = |\sA|$.
\item[\em (iii)] $s\in\partial\cS_{\cat_\sA}$ if and only if  $ \rank DS_\sA(C,X) < |\sA|$ for all  $(C,X)$ such that $S_\sA(C,X)=s$.
\item[\em (iv)] Suppose that $\cN_\sA < \cat_\sA$ and $\cS_k$ is closed for all $k=\cN_\sA,...,\cat_\sA.$ Then for each such $k$ there exists  $s\in\mathrm{int}\,\cS_{\cat_\sA}$ such that all $k$-atomic representing measures of $s$ are singular,  but $s$ has a regular representing measure with at least $k+1$ atoms.
\item[\em (v)] Suppose that $\cN_\sA < \cat_\sA$,  $\cS_k$ is closed for $k=\cat_\sA-1,\cat_\sA$ and $\cS_{\cat_\sA}\neq\rset^{|\sA|}$. If  $\rset^{|\sA|}\setminus\cS_{\cat_\sA-1}$ is path-connected,  there exists  $s\in\partial\cS_{\cat_\sA}$ such that $\cat_\sA(s)=\cat_\sA$.
\end{itemize}
\end{prop}
\begin{proof}


(i): Fix such a number  $k$ and assume $(\mathrm{int}\,\cS_k)\setminus\cS_{k-1}=\emptyset$. Then we have $\cS_{k-1}\supseteq\mathrm{int}\,\cS_{k}\supseteq\mathrm{int}\,\cS_{k-1}$. Taking the closure,   $\cS_{k-1} = \overline{\cS_{k-1}} \supseteq \overline{\cS_{k}} \supseteq \overline{\cS_{k-1}} = \cS_{k-1}$, so that  $\cS_k\subseteq\cS_{k-1}$ which contradicts (\ref{stricticl}). Thus, $(\mathrm{int}\,\cS_k)\setminus\cS_{k-1}\neq\emptyset$.

(ii): ``$\Leftarrow$'': Let $(C,X)$ be a full rank measure of $s$. Then a neighborhood $U$ of $(C,X)$ is mapped onto a neighborhood of $s$, that is, $s$ is an inner point.

``$\Rightarrow$'': Let $s$ be an inner point. Choose $\nu$ such that $S_\sA(\nu)$ has full rank. Since $s$ is an inner point, there exists  $\varepsilon>0$ such that $s':=s-\varepsilon\cdot S_\sA(\nu)$ is also an inner point. In particular, $s'$  is a moment sequence. Let $\mu'$ be a representing measure of $s'.$ Then $\mu = \mu' + \varepsilon\cdot\nu$ is a representing measure of $s$ and has full rank, since already $DS_\sA(\nu)$ has full rank.

(iii) follows from (ii).

(iv): Let $s\in\mathrm{int}\,\cS_k\subseteq\mathrm{int}\,\cS_{\cat_\sA}$.  By (i),  there exists $t\in (\mathrm{int}\,\cS_{\cat_\sA})\setminus\cS_k$ for all $k=\cN_\sA,...,\cat_\sA-1$. Then $[s,t]:=\{\lambda s + (1-\lambda) t \,|\, \lambda\in[0,1]\}\subseteq\mathrm{int}\,\cS_{\cat_\sA}$ by the convexity of $\cS_{\cat_\sA}$. Therefore, since $s\in\mathrm{int}\,\cS_k$ but $s\not\in\mathrm{int}\,\cS_k$, we have
\[\mathrm{int}\,\cS_{\cat_\sA}\cap\partial\cS_k \supseteq [s,t]\cap\partial\cS_k \neq \emptyset.\]
Hence there exists  $s\in\mathrm{int}\,\cS_{\cat_\sA}\cap\partial\cS_k$. Then all  $k$-atomic representing measures of $s$ are singular.  Otherwise, a full rank $k$-atomic measures implies that $s$ is an inner point of $\cS_k$. But, by (iv),  $s$ has a regular $l$-atomic measure with $l>k$.

(v): Let $s\in\mathrm{int}\,\cS_{\cat_\sA}\setminus\cS_{\cat_\sA-1}$ by (i) and $t\in\rset^{|\sA|}\setminus\cS_{\cat_\sA}$. Since  $s,t\in\rset^{|\sA|}\setminus\cS_{\cat_\sA-1}$, they are path-connected, so there exists a continuous path $\gamma:[0,1]\rightarrow\rset^{|\sA|}\setminus\cS_{\cat_\sA-1}$ with $\gamma(0)=s$ and $\gamma(1)=t$. But since $s=\gamma(0)\in\mathrm{int}\,\cS_{\cat_\sA}$, $t=\gamma(1)\not\in\mathrm{int}\,\cS_{\cat_\sA}$, and $\gamma([0,1])\subseteq\rset^{|\sA|}\setminus\cS_{\cat_\sA-1}$, we have $\gamma([0,1])\cap (\partial\cS_{\cat_\sA}\setminus\cS_{\cat_\sA-1})\neq\emptyset$. Therefore, $\partial\cS_{\cat_\sA}\setminus\cS_{\cat_\sA-1}\neq\emptyset$.
\end{proof}

In Sections \ref{monoone} and \ref{carzeropos} we derive upper bounds of $\cat_\sA$ by using Proposition \ref{propstoboundary} and the inequality (\ref{caboundary}). As (v) implies, this inequality can be strict, since  the Carath\'eodory number $\cat_\sA$ can be attained at a boundary point, see the following example.

\begin{exa}\label{exmp:CaratheodoryOnBoundary}
The (homogeneous) Motzkin polynomial
\[M(x,y,z) = z^6 + x^4 y^2 + x^2 y^4 - 3x^2 y^2 z^2\]
has the $6$ projective roots
\[\cZ(M) = \{(1,1,1),(1,1,-1),(1,-1,1),(1,-1,-1),(1,0,0),(0,1,0)\}.\]
We consider the truncated moment problem on the projective space $\pset(\rset^2) $ for 
\[\sB := \{z^6 + x^4 y^2 + x^2 y^4 - 3x^2 y^2 z^2,x^6,y^6,z^6,x^5 y, x^5 z, x^4 y z\}.\]
Clearly, $M\in\lin\sB$. Since $M$ is non-negative and has a discrete set of roots, $s = \sum_{\xi\in\cZ(M)} s_\sB(\xi)$ is a boundary point of the closed moment cone. The matrix
\[(s_\sB(\xi))_{\xi\in\cZ(M)} = \begin{pmatrix}
0 & 0 & 0 & 0 & 0 & 0\\
1 & 1 & 1 & 1 & 1 & 0\\
1 & 1 & 1 & 1 & 0 & 1\\
1 & 1 & 1 & 1 & 0 & 0\\
1 & 1 & -1 & -1 & 0 & 0\\
1 & -1 & 1 & -1 & 0 & 0\\
1 & -1 & -1 & 1 & 0 & 0
\end{pmatrix}\]
has rank $6$, i.e., the set $\{s_\sB(\xi)\}_{\xi\in\cZ(M)}$ is linearly independent. Hence $\cat_\sB(s) = 6$ by Theorem \ref{thm:LinIndepZeros}  and $\cat_\sB \leq 6 = |\sB| - 1$ by Theorem \ref{worstupperbound}. Thus, $\cat_\sB = \cat_\sB(s) = 6$, that is, the Carath\'eodory number is attained at the boundary moment sequence $s$.
\end{exa}

Next we derive an upper bound for the Carath\'eodory number in terms of zeros of positive elements of $\cA$. For the rest of this section we assume that $\cX$ is a closed subset of 
$\rset^n$ or $\pset(\rset^n)$ and $\sA \subseteq C^1(\cX,\rset)$. By the latter we mean that there exists an open subset $\cU$ of $\rset^n$ or $\pset(\rset^n)$ such that $\cX\subseteq \cU$ and $\sA \subseteq C^1(\cU,\rset).$

\begin{dfn}\label{defintionm_a}
 Let $\cM_\sA$ be the largest number $k$  obeying  the following property: \\
$(*)_k$:\, There  exist $f\in \cA$ and  $x_1,\dots,x_k\in \cZ(f)$ such that $f(x)\geq 0$ on $\cX$ and $\{s_\sA(x_i)\}_{i=1,...,k}$ is linearly independent ($DS_{k,\sA}((1,...,1),(x_1,...,x_k))$ does not have full rank).
\end{dfn}

From the definition it is clear that $\cM_\sA$ is the largest dimension an exposed face of $\cS_\sA$.

\begin{prop}\label{propboundmsa}
For each   $s\in \partial \cS_\sA\cap \cS_\sA$ we have $\cat_\sA(s)\leq \cM_\sA.$
\end{prop}
\begin{proof}
In this proof we abbreviate $N:=\cat_\sA(s)$.
Let $\mu = \sum_{i=1}^N c_i \delta_{x_i}$ be an $N$-atomic representing measure of $s$.
Since $s\in \partial \cS_\sA\cap \cS_\sA$, there exists $f\in \cA, f\neq 0,$ such that  $f(x)\geq 0$ on $\cX$ and $L_s(f)=0$. From the latter it follows    that  $\supp\mu\subseteq\cZ(f)$ and hence $x_1,\dots,x_N\in \cZ(f).$ Further, by Proposition \ref{prop:moreOnMomentCone}(iii),  $s\in \partial \cS_\sA\cap \cS_\sA$  implies that $DS_{N,\sA}(C,X)$ does not have full rank $|\sA|$. Since $c_i>0$ for all $i$, we have ${\rm rank}\, DS_{N,\sA}(C,X)=\rank DS_{N,\sA}((1,\dots,1),X).$ 
Finally, by Theorem \ref{thm:LinIndepZeros} the set $\{s_\sA(x_i)\}_{i=1,...,N}$ is linearly independent. Thus,  property $(*)_N$  in Definition \ref{defintionm_a} holds, so that   $\cat_\sA(s)=N\leq \cM_\sA$.
\end{proof}

\begin{thm}
Suppose that $\cX$ is a compact subset of $\rset^n$ or $\pset(\rset^n)$, condition (\ref{cond+}) is satisfied, and $\sA\subseteq C^1(\cX,\rset)$. Then
\[\cat_\sA \leq \cM_\sA + 1.\]
\end{thm}
\begin{proof}
The assumptions of this theorem ensure that Proposition \ref{propstoboundary} applies. Hence the assertion follows by combining  Proposition \ref{propboundmsa} with the inequality (\ref{caboundary}). 
\end{proof}

\section{Carath\'eodory Numbers: One-dimensional Monomial Case}\label{monoone}

For the one-dimensional truncated moment problem the number $\cN_\sA$ can be calculated from the formula for the Vandermonde determinant.

\begin{lem}\label{lem:onedimNA}
Let  $\sA := \{1,x,...,x^n\}$, where  $n\in \nset$.
\begin{itemize}
\item[\em (i)] If $n= 2k-1, k\in \nset$, then
\begin{equation}\label{eq:oneDimDet1}
\det DS_{k,\sA} = c_1\cdots  c_k \cdot \prod_{1\leq i < j \leq k} (x_j - x_i)^4.
\end{equation}
\item[\em (ii)] If $n=2k, k\in \nset$, then
\begin{equation}\label{eq:oneDimDet2}
\det (DS_{k-1,\sA},s_\sA(x_k)) = c_1\cdots  c_{k-1} \cdot \prod_{1\leq i < j \leq k-1} (x_j - x_i)^4 \cdot \prod_{i=1}^{k-1} (x_k - x_i)^2.
\end{equation}
\item[\em (iii)] $\cN_\sA = \left\lfloor\frac{n}{2}\right\rfloor+1 =\left\lceil\frac{n+1}{2}\right\rceil$.
\end{itemize}
\end{lem}

\begin{proof}
We carry out the proofs in the odd case $n = 2k - 1$. The even case $n=2k$ is derived in a similar manner. 

We have  $\partial_{c_i} S_k(C,X) = s_\sA(x_i)$ and $\partial_{x_i} S_k(C,X) = c_i s_\sA'(x_i)$. Therefore,
\begin{align}\label{DS_k}
\det DS_{k,\sA}=c_1\cdots c_k\det(s_\sA(x_1),s_\sA'(x_1),...,s_\sA(x_k),s_\sA'(x_k))
\end{align}
and we compute
\begin{align*}
&\det(s_\sA(x_1),s_\sA'(x_1),...,s_\sA(x_k),s_\sA'(x_k))\\
 &= \lim_{h_1\rightarrow 0} ... \lim_{h_k\rightarrow 0} \det\left(s_\sA(x_1),\frac{s_\sA(x_1+h_1)-s_\sA(x_1)}{h_1},...,\frac{s_\sA(x_k+h_k)-s_\sA(x_k)}{h_k}\right)\\
&= \lim_{h_1\rightarrow 0} ... \lim_{h_k\rightarrow 0} \frac{\det(s_\sA(x_1),s_\sA(x_1+h_1),...,s_\sA(x_k+h_k))}{h_1\cdots h_k}\\
&= \lim_{h_1\rightarrow 0} ... \lim_{h_k\rightarrow 0} \frac{\prod_{i=1}^k\left( h_i \prod_{j=i+1}^k (x_j{+}h_j {-} x_i)(x_j{-}x_i)(x_j {-} x_i {-} h_i) (x_j {+} h_j {-} x_i {-} h_i) \right)}{h_1\cdots h_k}\\
&= \prod_{1\leq i < j \leq k} (x_j - x_i)^4.
\end{align*}
Combined with (\ref{DS_k}), this yields (\ref{eq:oneDimDet1}).

We choose the numbers $x_i$ pairwise different and all $c_i$ positive. Then the determinants in (\ref{eq:oneDimDet1}) and (\ref{eq:oneDimDet2}) are non-zero. Hence $\det{DS_{k,\sA}}\neq 0$ and therefore $\cN_\sA = k = \left\lfloor\frac{n}{2}\right\rfloor + 1$.
\end{proof}

\begin{exa}\label{exmp:one-dimCatnumber}
H.\ Richter \cite{richte57} has shown that for the one-dimensional truncated moment problem $\sA=\{1,x,...,x^d\}$ the Carath\'eodory number is $\cat_{\sA} =  \left\lceil\frac{d+1}{2}\right\rceil$. This result will also follow from Theorem \ref{thm:onedimCara} below. If we take this equality for granted and  combine it with Lemma \ref{lem:onedimNA}(iii), then we obtain 
\begin{align*}
\cN_{\sA} =  \left\lfloor\frac{d}{2}\right\rfloor+1 =\left\lceil\frac{d+1}{2}\right\rceil = \cat_{\sA}.
\end{align*}
\end{exa}

Now we turn to the general case and assume that
\begin{align}\label{sageneralcase}
{\sA}=\{x^{d_1},...,x^{d_m}\}, ~~\text{where}~~ 0\leq d_1 < d_2 <...<d_m,~~d_1,\dots,d_m\in \nset_0. 
\end{align}
Then we compute
\begin{align*}
 f_{\sA}(x_1,&...,x_m) :=\det(s_{\sA}(x_1),...,s_{\sA}(x_m)) = |(x_i^{d_j})_{i,j=1,...,m}|\nonumber \\ &= \begin{vmatrix}
x_1^{d_1} & x_2^{d_1} & \cdots & x_m^{d_1}\\
x_1^{d_2} & x_2^{d_2} & \cdots & x_m^{d_2}\\
\vdots & \vdots & \ddots & \vdots\\
x_1^{d_m} & x_2^{d_m} & \cdots & x_m^{d_m}
\end{vmatrix}
= (x_1\cdots x_m)^{d_1} \cdot
\begin{vmatrix}
1 & 1 & \cdots & 1\\
x_1^{d_2-d_1} & x_2^{d_2-d_1} & \cdots & x_m^{d_2-d_1}\\
\vdots & \vdots & \ddots & \vdots\\
x_1^{d_m-d_1} & x_2^{d_m-d_1} & \cdots & x_m^{d_m-d_1}
\end{vmatrix} .
\end{align*}
From the latter equation it follows that each linear polynomial $x_j - x_i$,  $j\neq i$, divides the polynomial $f_{\sA}$. Hence there exists a polynomial $p_{\sA}$ such that 
\begin{align}\label{eq:determantproductOneDim}
f_{\sA}(x_1,...,x_m) & =|(x_i^{d_j})_{i,j=1,...,m}|\nonumber\\ &= (x_1\cdots  x_m)^{d_1}  \prod_{1\leq i < j \leq m} (x_j {-} x_i) \cdot p_{\sA}(x_1,...,x_m)
\end{align}
The polynomial $p_{\sA}$ is uniquely determined by (\ref{eq:determantproductOneDim}). It  is homogeneous with degree
\begin{equation}
\deg p_{\sA}= d_1 + \dots + d_m - m d_1 - \frac{m(m-1)}{2}.
\end{equation}
Such polynomials $p_{\sA}$ are called Schur polynomials. They are well studied in the literature, see e.g.\ \cite{macdonSymFuncHallPoly}. For these Schur polynomials  it is known that 
\begin{equation}\label{Young}
p_{\sA}(x_1,...,x_m) = \sum_\alpha x^\alpha,
\end{equation}
where $\alpha$ ranges over some Young tableaux. In particular, (\ref{Young}) implies that all non-zero coefficients of $p_{\sA}$ are positive.

\begin{exa}\label{exmps:SchurPolynomials}
\begin{enumerate}
\item ${\sA}=\{x,x^4,x^7\}$. Then we compute 
\[\det(s_{\sA}(x_1),s_{\sA}(x_2),s_{\sA}(x_3)) = x_1 x_2 x_3 \prod_{1\leq i < j\leq 3}(x_j - x_i)\cdot p_{\sA}(x_1,x_2,x_3),\]
where
\[p_{\sA}(x_1,x_2,x_3) = (x_1^2 + x_1 x_2 + x_2^2)(x_1^2 + x_1 x_3 + x_3^2)(x_2^2 + x_2 x_3 + x_3^2).\]

\item ${\sA}=\{x,x^2,x^6\}$. Then 
\[\det(s_{\sA}(x_1),s_{\sA}(x_2),s_{\sA}(x_3)) = x_1 x_2 x_3 \prod_{1\leq i < j\leq 3}(x_j - x_i)\cdot p_{\sA}(x_1,x_2,x_3),\]
where
\begin{align*}
p_{\sA}(x_1,x_2,x_3) =&\ x_1^3 + x_1^2 x_2 + x_1^2 x_3 + x_1 x_2^2 + x_1 x_2 x_3
+ x_1 x_3^2 + x_2^3\\ &+ x_2^2 x_3 + x_2 x_3^2 + x_3^3.
\end{align*}

\item ${\sA}=\{1,x,x^2,x^6\}$. Then 
\[\det(s_{\sA}(x_1),s_{\sA}(x_2),s_{\sA}(x_3),s_{\sA}(x_4)) = \prod_{1\leq i < j\leq 3}(x_j - x_i)\cdot p_{\sA}(x_1,x_2,x_3,x_4)\]
with
\begin{align*}
p_{\sA}(x_1,x_2,x_3,x_4) = &\ x_1^3 + x_1^2 x_2 + x_1^2 x_3 + x_1^2 x_4 + x_1 x_2^2 + x_1 x_2 x_3 + x_1 x_2 x_4\\
&+ x_1 x_3^2 + x_1 x_3 x_4 + x_1 x_4^2 + x_2^3 + x_2^2 x_3 + x_2^2 x_4 + x_2 x_3^2\\
&+ x_2 x_3 x_4 + x_2 x_4^2 + x_3^3 + x_3^2 x_4 + x_3 x_4^2 + x_4^3.
\end{align*}

\item $\sA = \{1,x^2,x^3,x^5,x^6\}$. Then
\[\det((s_{\sA}(x_i))_{i=1}^5) = \prod_{1\leq i < j\leq 5}(x_j - x_i)\cdot p_{\sA}(x_1,x_2,x_3,x_4,x_5)\]
with
\[p_{\sA}(x_1,x_2,x_3,x_4) = \sum_{\alpha\in \Omega} x^\alpha + 3\sum_{\alpha\in \Phi} x^\alpha,\]
$\Omega = \{\text{all permuations of}\ (2,2,1,1,0)\}$, $\Phi = \{\text{all permuations of}\ (2,1,1,1,1)\}$.
\end{enumerate}
\end{exa}

\begin{dfn}\label{defintionq_a}
Assume that $\sA$ is as in  (\ref{sageneralcase}) and  $p_{\sA}$ is defined by  (\ref{eq:determantproductOneDim}). Set
\[q_{\sA}(x_1,...,x_k):= p_{\sA}(x_1,x_1,...,x_k,x_k)\]
if $m=2k$ is even and
\[q_{{\sA},i}(x_1,...,x_k):= p_{\sA}(x_1,x_1,...,x_{i-1},x_{i-1},x_i,x_{i+1},x_{i+1},...,x_k,x_k)\]
for all $i=1,...,k$ if $m = 2k-1$ is odd.
\end{dfn}

\begin{lem}\label{symlemma0}
\begin{enumerate}
\item[\em (i)] If $m$ is even then $q_{\sA}$ is symmetric.
\item[\em (ii)] If $m$ is odd then $q_{\sA,i}(x_1,...,x_k) = q_{\sA,k}(x_1,...,x_{i-1},x_{i+1},...,x_k,x_i)$ for all $i=1,...,k$.
\end{enumerate}
\end{lem}
\begin{proof}
(i): Since the Schur polynomial $p_{\sA}$ is symmetric, so is $q_{\sA}$.

(ii): We derive
\begin{align*}
q_{\sA,i}(x_1,...,x_k)
&= p_{\sA}(x_1,x_1,...,x_{i-1},x_{i-1},x_i,x_{i+1},x_{i+1},...,x_k,x_k)\\
&= p_{\sA}(x_1,x_1,...,x_{i-1},x_{i-1},x_{i+1},x_{i+1},...,x_k,x_k,x_i)\\
&= q_{\sA,k}(x_1,...,x_{i-1},x_{i+1},...,x_k,x_i).\qedhere
\end{align*}
\end{proof}

In the odd  case it suffices to prove formula (\ref{eq:oneDimDet2general}) below. All other determinants are then obtained by interchanging variables and Lemma \ref{symlemma0}(ii).

\begin{lem}\label{lem:onedimNA2}
Suppose that $\sA$ is of the form (\ref{sageneralcase}).
\begin{itemize}
\item[\em (i)]  If $m= 2k$ is even, then

\begin{equation}\label{eq:oneDimDet1general}
\begin{split}
& \det DS_{k,{\sA}}(c_1,...,c_k,x_1,...,x_k)=\\
& c_1\cdots c_k \cdot (x_1\cdots x_k)^{2d_1} \prod_{1\leq i < j \leq k} (x_j - x_i)^4 \cdot q_{\sA}(x_1,...,x_k).
\end{split}
\end{equation}

\item[\em (ii)] If $m=2k-1$ is odd, then
\begin{equation}\label{eq:oneDimDet2general}
\begin{split}
& \det (DS_{k-1,{\sA}},s_{\sA}(x_k))=\\
& c_1\cdots c_{k-1}\cdot (x_1\cdots x_{k-1})^{2d_1} x_k^{d_1} \cdot \prod_{1\leq i < j \leq k-1} (x_j - x_i)^4 \cdot \prod_{i=1}^{k-1} (x_k - x_i)^2 \cdot q_{{\sA},k}(x_1,...,x_k).
\end{split}
\end{equation}

\item[\em (iii)] $\cN_{\sA} = \left\lceil \frac{m}{2} \right\rceil$.
\end{itemize}
\end{lem}
\begin{proof}
(i): Clearly, $\partial_{c_i} S_k(C,X) = s_{\sA}(x_i)$ and $\partial_{x_i} S_k(C,X) = c_i s_{\sA}'(x_i)$. By the linearity of the determinant, the factor $c_1 \cdots c_k$ can be taken out, so we can assume without loss of generality that $c_1=...=c_k=1$. 

Let $m=2k$. We proceed in a similar manner as in the proof of Lemma \ref{lem:onedimNA}. Using (\ref{eq:determantproductOneDim}) we derive
\begin{align*}
&\det(s_\sA(x_1),s_\sA'(x_1),...,s_\sA'(x_k))\\
 &= \lim_{h_1\rightarrow 0} ... \lim_{h_k\rightarrow 0} \det\left(s_\sA(x_1),\frac{s_\sA(x_1+h_1)-s_\sA(x_1)}{h_1},...,\frac{s_\sA(x_k+h_k)-s_\sA(x_k)}{h_k}\right)\\
&= \lim_{h_1\rightarrow 0} ... \lim_{h_k\rightarrow 0} \frac{\det(s_\sA(x_1),s_\sA(x_1+h_1),...,s_\sA(x_k+h_k))}{h_1\cdots h_k}\\
&= \lim_{h_1\rightarrow 0} ... \lim_{h_k\rightarrow 0} \frac{f_{\sA}(x_1,x_1+h_1,x_2,x_2+h_2,...,x_k+h_k)}{h_1 \cdots h_k}\\&= \lim_{h_1\rightarrow 0} ... \lim_{h_k\rightarrow 0} \frac{\prod_{i=1}^k\left( h_i \prod_{j=i+1}^k (x_j{+}h_j {-} x_i)(x_j{-}x_i)(x_j {-} x_i {-} h_i) (x_j {+} h_j {-} x_i {-} h_i) \right)}{h_1\cdots h_k}\\ &\quad \quad \times (x_1 (x_1+h_1) \cdots x_k (x_k+h_k))^{d_1}~ p_{\sA}(x_1,x_1+h_1,x_2,x_2+h_2,...,x_k,x_k+h_k)\\
&= (x_1 \cdots x_k)^{2d_1} \cdot \prod_{i<j} (x_j - x_i)^4 \cdot p_{\sA}(x_1,x_1,x_2,x_2,...,x_k)\\
&= (x_1 \cdots x_k)^{2d_1} \cdot \prod_{i<j} (x_j - x_i)^4 \cdot q_{\sA}(x_1,...,x_k).
\end{align*}

(ii): The proof in the odd case $n=2k-1$ is similar.

(iii): Since  $q_{\sA}$ is not the zero polynomial and all nonzero coefficients are positive,  there are $x_1,...,x_k$ such that $\det(DS_{k,\sA})(x_1,...,x_k)\neq 0.$ Then $DS_{k,\sA}$ has full rank, so that  $\cN_\sA=k=\lceil m/2\rceil .$
\end{proof}

Now we turn to the homogeneous case and set 
\begin{align}\label{definiions}
{\sB} = \{x^{d_1} y^{d_m-d_1},\dots,x^{d_m}\},~~ \text{where}~ 0\leq d_1 < d_2 < ... < d_m,\, d_i\in \nset_0.
\end{align}

\begin{exa}\label{exmps:SchurPolynomialsHom}
\begin{enumerate}
\item In the case $\sB = \{x y^7, x^4 y^4,x^7 y\}$ we have
\[\det((s_{\sB}(x_i,y_i))_{i=1}^3) = x_1 y_1 x_2 y_2 x_3 y_3 \cdot \prod_{1\leq i< j\leq 3} (x_j y_i - x_i y_j)\cdot p_{\sB}(x_1,y_1,x_2,y_2,x_3,y_3)\]
with
\[p_{\sB}(x_1,y_1,x_2,y_2,x_3,y_3) = \prod_{1\leq i < j \leq 3}(x_i^2 y_j^2 + x_i y_i x_j y_j + x_j^2 y_i^2).\]

\item $\sB = \{x y^5, x^2 y^4, x^6\}$. Then we have
\[\det((s_{\sB}(x_i,y_i))_{i=1}^3) = x_1 x_2 x_3 \cdot \prod_{1\leq i< j\leq 3} (x_j y_i - x_i y_j)\cdot p_{\sB}(x_1,y_1,x_2,y_2,x_3,y_3)\]
with
\begin{align*}
p_{\sB}(x_1,x_2,x_3) =&\ x_1^3 y_2^3 y_3^3 + x_1^2 y_1 x_2 y_2^2 y_3^3 + x_1^2 y_1 y_2^3 x_3 y_3^2 + x_1 y_1^2 x_2^2 y_2 y_3^3\\
&+ x_1 y_1^2 x_2 y_2^2 x_3 y_3^2 + x_1 y_1^2 y_2^3 x_3^2 y_3 + y_1^3 x_2^3 y_3^3 + y_1^3 x_2^2 y_2 x_3 y_3^2\\
&+ y_1^3 x_2 y_2^2 x_3^2 + y_1^3 y_2^3 x_3^3.
\end{align*}
\end{enumerate}
\end{exa}

\begin{dfn}
For even $m=2k$ we define
\begin{equation}\label{defq_sB}
q_{\sB}(x_1,y_1,\dots,x_k,y_k):= 
(y_1\cdots y_k)^{2(d_m-d_1-m) +3} q_{\sA}\left(\frac{x_1}{y_1},\cdots,\frac{x_k}{y_k}\right).
\end{equation}
For odd $m=2k-1$ we set
\begin{equation}\label{defqbkpoly}
q_{{\sB},k}(x_1,y_1,...,x_k,y_k) := (y_1\cdots y_{k-1})^{2d_m-2d_1-3m+6} y_k^{d_m-d_1-m+1} \,  q_{{\sA},k}\left(\frac{x_1}{y_1},\dots,\frac{x_k}{y_k}\right).
\end{equation}
\end{dfn}

\begin{lem}\label{lemmainhosB}
Let $\sB$ be of the form (\ref{definiions}).
\begin{itemize}
\item[\em (i)] If $m=2k$ is even, then
\begin{equation*}
\begin{split}
 &\det D_{c,x} S_{k,\sB}(c_1,...,c_k,x_1,y_1,...,x_k,y_k)\\
 &= c_1\cdots c_k \cdot (x_1\cdots x_k)^{2d_1}\cdot \prod_{1\leq i < j \leq k} (x_j y_i - x_i y_j)^4 \cdot q_{\sB}(x_1,y_1,\dots,x_k,y_k).
\end{split}
\end{equation*}

\item[\em (ii)] If  $m=2k-1$ is odd, then
\begin{equation*}
\begin{split}
&\det (D_{c,x}S_{k-1,{\sB}}(c_1,...,c_{k-1},x_1,y_1,...,y_{k-1}),s_{\sB}(x_k,y_k))= c_1\cdots c_{k-1} (x_1\cdots x_{k-1})^{2d_1}\\
&\qquad\times x_k^{d_1} \prod_{1\leq i < j \leq k-1} (x_j y_i {-} x_i y_j)^4\cdot  \prod_{i=1}^{k-1} (x_k y_i {-} x_i y_k)^2 \cdot q_{{\sB},k}(x_1,...,x_k).
\end{split}
\end{equation*}

\item[\em (iii)] $q_\sB$ in (\ref{defq_sB}) and $q_{\sB,k}$ in (\ref{defqbkpoly}) are in $\nset_0[x_1,y_1,...,x_k,y_k]$.
\item[\em (iv)] $\cN_{\sB} = \left\lceil\frac{m}{2}\right\rceil$.
\end{itemize}
\end{lem}
\begin{proof}
(i): Again it suffices to prove the formulas in the case $c_1=...=c_k=1$. We set  $u = \frac{x}{y}$ and $u_i = \frac{x_i}{y_i}$. Using the relation  $\partial_x = y^{-1} \partial_u$ and equations (\ref{eq:oneDimDet1}) and (\ref{defq_sB}) we compute 
\begin{align*}
&\det D_{c,x} S_{k,{\sB}}(c_1,...,c_k,x_1,y_1,...,x_k,y_k)\\
&= \det (s_{\sB}(x_1,y_1),\partial_x s_{\sB}(x_1,y_1),...,s_{\sB}(x_k,y_k),\partial_x s_{\sB}(x_k,y_k))\\
&= (y_1\cdots  y_k)^{2d_m} \cdot \det(s_{\sA}(u_1),\partial_x s_{\sA}(u_1),...,s_{\sA}(u_k),\partial_x s_{\sA}(u_k))\\
&= (y_1\cdots  y_k)^{2d_m-1} \cdot \det(s_{\sA}(u_1),\partial_u s_{\sA}(u_1),...,s_{\sA}(u_k),\partial_u s_{\sA}(u_k))\\
&= (y_1\cdots y_k)^{2d_m-1} (u_1\cdots u_k)^{2d_1} \cdot \prod_{1\leq i < j\leq k} (u_j-u_i)^4 \cdot q_{\sA}(u_1,...,u_k)\\
&=  (x_1\cdots x_k)^{2d_1}(y_1\cdots y_k)^{2(d_m-d_1-m)+3} \prod_{1\leq i < j\leq k} (x_j y_i-x_i y_j)^4 \cdot q_{\sA}\left(\frac{x_1}{y_1},...,\frac{x_k}{y_k}\right)\\
&=  (x_1\cdots x_k)^{2d_1} \prod_{1\leq i < j\leq k} (x_j y_i-x_i y_j)^4  \cdot q_{{\sB}}(x_1,y_1,...,x_k,y_k).
\end{align*}

(ii): We proceed in a similar manner and derive
\begin{align*}
&\det (s_{\sB}(x_1,y_1),\partial_x s_{\sB}(x_1,y_1),...,s_{\sB}(x_k,y_k))\\
&= (y_1\cdots  y_{k-1})^{2d_m} \cdot y_k^{d_m}\cdot \det(s_{\sA}(u_1),\partial_x s_{\sA}(u_1),...,s_{\sA}(u_k))\\
&= (y_1\cdots  y_{k-1})^{2d_m-1} \cdot y_k^{d_m}\cdot \det(s_{\sA}(u_1),\partial_u s_{\sA}(u_1),...,s_{\sA}(u_k))\\
&= (u_1\cdots u_{k-1})^{2d_1} u_k^{d_1} (y_1\cdots y_{k-1})^{2d_m-1} y_k^{d_m} \prod_{1\leq i < j\leq k-1} (u_j-u_i)^4 \prod_{i=1}^{k-1} (u_k-u_i)^2\\
&\qquad\times q_{\sA,k}(u_1,...,u_k)\\
&=  (x_1\cdots x_{k-1})^{2d_1} x_k^{d_1} (y_1\cdots y_{k-1})^{2d_m-2d_1-3m+6} y_k^{d_m-d_1-m+1} \prod_{1\leq i < j\leq k-1} (x_j y_i-x_i y_j)^4\\
&\qquad\times \prod_{i=1}^{k-1} (x_k y_i - x_i y_k)^2 \cdot q_{\sA,k}\left(\frac{x_1}{y_1},...,\frac{x_k}{y_k}\right)\\
&=  (x_1\cdots x_{k-1})^{2d_1} x_k^{d_1} \prod_{1\leq i < j\leq k} (x_j y_i-x_i y_j)^4  \cdot \prod_{i=1}^{k-1} (x_k y_i - x_i y_k)^2 \cdot q_{{\sB},k}(x_1,y_1,...,x_k,y_k).
\end{align*}

(iii): First we show that $q_\sB$ and $q_{\sB,k}$ are polynomials. That they are polynomials in $x_1,...,x_k$ is clear, since $q_\sA$ and $q_{\sA,k}$ are polynomials in the coordinates and all $x_i$ appear only with non-negative exponent in the definitions. Therefore, it suffices to show that they are also polynomials in all $y_i$. We will only prove the statement for $q_\sB$, for $q_{\sB,k}$ the same chain of arguments holds.

Assume  the contrary. Then  $q_\sB$ contains a term with $y_k^{-l}$ for some $l>0$ with non-zero coefficient.  Let $l$ be  the largest  such $l$ and let $f(x,y):=\sum_i a_i x^{\alpha_i} y^{\beta_i}$ be the factor of $y_k^{-l}$  in $q_\sB$. Since $f$ is non-zero by assumption,  there are  $Z = (x_1,y_1,...,x_k,y_k) \in \rset^{2k}$ and  $\varepsilon>0$ such that $f$ is non-zero on  the ball $B_\varepsilon(Z)$ centered at $Z$ with radius $\varepsilon$.

On the other hand, we expand $\prod_{1\leq i < j\leq k} (x_j y_i-x_i y_j)^4$ and let $g(x,y) :=\sum_i b_i x^{\gamma_i} y^{\delta_i}$ be the sum of all terms therein which contain no $y_k$.  Then $g$ is a polynomial in all $x_i$ and $y_i$ and $g$ is not the zero-polynomial. Hence $g$ is not identically zero on $B_\varepsilon(Z)$ and so is
\[fg = \sum_{i,j} a_i b_j x^{\alpha_i + \gamma_j} y^{\beta_i+\delta_i}.\]
From the Laplace expansion it follows that the determinant
\[\det (s_{\sB}(x_1,y_1),\partial_x s_{\sB}(x_1,y_1),...,s_{\sB}(x_k,y_k),\partial_x s_{\sB}(x_k,y_k))\tag{$*$}\]
is a polynomial in  $x_i, y_j$. Further, $fg$ appears in the expansion of the product
\[(x_1\cdots x_k)^{2e_1} \prod_{1\leq i < j\leq k} (x_j y_i-x_i y_j)^4  \cdot q_{{\sB}}(x_1,y_1,...,x_k,y_k)\tag{$**$}\]
and by the maximality of $l$ it does not cancel. Hence ($*$) does not contain a term  with $y_k^{-l}$ but ($**$) does. Since both are equal by (i), we get a contradiction.  Thus, $q_\sB$ is a polynomial in all $y_i$.

It remains to show that all coefficients of $q_\sB$ are in $\nset_0$. Since $q_\sA$  comes from the Schur polynomial $p_\sA$ (see  (\ref{Young})), its coefficients are in $\nset_0$. This is not changed by multiplication with $(y_1\cdots y_k)^{2(d_m-d_1-m) +3}$, so $q_\sB$ has $\nset_0$-coefficients as well.

(iv): By (iii) all nonzero coefficients of $q_{\sB}$ are positive integers. Hence, by (i) and (ii), we can find real numbers $x_1,y_1,\dots,x_k,y_k$ such that the corresponding determinants are non-zero. Hence $\cN_{\sB} = k=\left\lceil\frac{m}{2}\right\rceil$.
\end{proof}

\begin{exa}\label{exmps:SchurPolynomialsHom2}
\begin{enumerate}
\item Let $\sB = \{x y^7, x^4 y^4, x^7 y\}$. Then we have
\begin{equation*}
\begin{split}
&\det(s_{\sB}(x_1,y_1),\partial_x s_{\sB}(x_1,y_1),s_{\sB}(x_2,y_2))\\ &= 3x_1^4 y_1^5 x_2 y_2 (x_1 y_2 - x_2 y_1)^2 (x_1^2 y_2^2 + x_1 x_2 y_1 y_2 + x_2^2 y_1^2)^2.
\end{split}
\end{equation*}

\item $\sB = \{x y^5, x^2 y^4, x^6\}$. Then
\begin{equation*}
\begin{split}
&\det(s_{\sB}(x_1,y_1),\partial_x s_{\sB}(x_1,y_1),s_{\sB}(x_2,y_2))\\
&= x_1^2 x_2 y_1^4 (x_1 y_2 - x_2 y_1)^2(4 x_1^3 y_2^3 + 3x_1^2 x_2 y_1 y_2^2 + 2 x_1 x_2^2 y_1^2 y_2 + x_2^3 y_1^3).
\end{split}
\end{equation*}
\end{enumerate}
\end{exa}

The following theorem is the main result of this section. It gives sufficient conditions for the validity of formula (\ref{qbnb}) concerning the Carath\'eodory number $ \cat_{\sB}$.

\begin{thm}\label{thm:onedimCara}
Let $m,d_1,d_2,\dots,d_m,d\in\nset$ be such that $0= d_1 <...<d_m=2d$, put $\sA=\{1,x^{d_2},...,x^{d_m}\}$, ${\sB}=\{y^{2d},x^{d_2}y^{2d-d_2},...,x^{d_{m-1}}y^{2d-d_{m-1}},x^{2d}\}$, and $\cZ := \cZ(q_{{\sA}})$ if $m$ is even or $\cZ := \cZ(q_{\sA,1})\cap ...\cap \cZ(q_{{\sA},k})$ if $m=2k-1$ is odd, where $q_{\sA}$ and $ q_{{\sA},j}$ are as in Definition \ref{defintionq_a}. Suppose that
\begin{equation}\label{condition:one-dim-Caratheodory}
(x_1,...,x_k)\in\cZ\folgt\exists i\neq j: x_i = x_j.
\end{equation}
Then
\begin{align}\label{qbnb}
\cat_\sA =\cat_{\sB} = \cN_{\sA} = \cN_\sB =\left\lceil \frac{m}{2}\right\rceil.
\end{align}
\end{thm}
\begin{proof}
Recall that $\cS_\sA$ and $\cS_\sB$ denote the moment cones of $\sA$ and $\sB$, respectively. We  set $\partial^* \cS_\sA := \partial\cS_\sA\cap\cS_\sA$.

By Lemmas \ref{lem:onedimNA2}(iii) and \ref{lemmainhosB}(iv) we have $\cN_{\sA} = \cN_{\sB} = \left\lceil \frac{m}{2}\right\rceil$. Further,  $\cN_{\sA}\leq \cat_\sA$ and $\cN_{\sA}\leq \cat_\sB$ by Theorem \ref{densethm}. Therefore, it suffices to show that $\cat_\sA\leq\cN_\sA$ and 
$\cat_\sB\leq\cN_\sA$.

First we prove that $\cat_\sB\leq\cN_\sA.$ 

Let $s\in\cS_\sB$. Since $\cX=\pset(\rset)$ is compact and condition (\ref{cond+}) is satisfied (with $e(x,y):=x^{2d}+y^{2d}\in {\cB}$),  Proposition \ref{propstoboundary} applies with $x=(1,0).$ Hence the supremum $c_s(1,0):=\sup\, \{ c\in \rset: s- c s_{\sB}(1,0) \in \cS_\sB  \}$ is attained and $s':= s - c_s(1,0) s_{\sB}(1,0)\in\partial\cS_\sB$. By Proposition \ref{prop:moreOnMomentCone}(iii) all representing measures $(C',X')$ of $s'$ are singular. They do not contain $(1,0)$ as an atom. (Indeed,  otherwise $c_s(1,0)$ could be increased which contradicts to the maximality of $c_s(1,0)$.) Since the polynomials of $\sB$ are homogeneous, we can assume without loss of generality  that\ $X_i' = (x_i',1)$ with $x_i'$ pairwise different, say $x_1' < x_2' < ... < x_l'$, and $s'\in\partial^*\cS_\sA$, i.e., $s'$ is a boundary moment sequence of $\cS_\sA$. But from (\ref{condition:one-dim-Caratheodory}) and Lemma \ref{lemmainhosB},(i) and (ii), it follows that $l<\cN_\sA$, that is, $\cat_{\sB}(s) \leq l + 1 \leq \cN_\sB = \cN_\sA$. This completes the proof  of the inequality $\cat_\sB\leq \cN_\sA$.

Next we show that $\cat_\sA\leq\cN_\sA.$ 

If $s\in\partial^*\cS_\sA$, then   $\cat_\sA(s)<\cN_\sA$ by the preceding proof. Now let   $s\in\mathrm{int}\, \cS_\sA = \mathrm{int}\, \cS_\sB$. Then $\cat_\sB(s)\leq\cN_\sA$  by the preceding paragraph  and it suffices  to show that $s$ has an at most $\cN_\sA$-atomic representing measure which does not have an atom at $(1,0)$. We choose $\varepsilon > 0$ such that\ $B_\varepsilon(s)\subseteq \mathrm{int}\ \cS_\sA$.   

Let $c_t(x)$ be defined by (\ref{supcx}).  
Since $t\mapsto L_t$ is a continuous map of $\rset^m\to \cA^*$, $t\mapsto L_t(e)$ is continuous. Hence,  $c_t(x)\leq e(x)^{-1}L_t(e)$ (by Proposition \ref{propstoboundary}) is bounded from above on $ \overline{B_\varepsilon(s)}$.  Then the supremum $C$ of $c_t(1,0)$ on $\overline{B_\varepsilon(s)}$ is finite. Let
\[ T := \bigcup_{c\in[0,C+1]} \overline{B_\varepsilon(s-c\cdot s_\sB(1,0))}\]
be the $\varepsilon$-tube around the line $\gamma := s-[0,C+1]\cdot s_\sB(1,0)$. Write $T = T_1 \cup T_2 \cup T_3$ with $T_2 := T\cap\partial\cS_\sB$, $T_1 := T\cap\mathrm{int}\ \cS_\sB$, and $T_3 := T\setminus (T_1\cup T_2)$, i.e., $T_1$ is the part inside $\cS_\sB$, $T_3$ is the part outside $\cS_\sB$, and $T_2$ is the boundary part of $\cS_\sB$ in $T$. Since $\cS_\sB$ is closed and convex, $T_2$ is closed and every path in $T$ starting in $T_1$ and ending in $T_3$ contains at least one point in $T_2$. By construction, $t' := t - c_t(1,0) s_\sB(1,0)\in T_2$ for all $t\in T_1$ and no representing measure of $t'$ contains $(1,0)$ as an atom, i.e., $T_2\subset \partial^*\cS_\sA$. Then $\gamma = s - [0,1]\cdot (C+1) s_\sB(1,0) \subset T$ and $s\in\gamma\cap T_1$ and $s - (C+1) s_\sB(1,0)\in\gamma\cap T_3$, so that $s' = s - c_s(1,0) s_\sB(1,0)\in T_2$. Since $s_\sB$ is continuous and $C<\infty$ there is a $\delta> 0$ such that
\[\|(s - (C+1) s_\sB(1,0)) - (s - (C+1) s_\sB(1,\delta)) \| = (C+1)\| s_\sB(1,0) - s_\sB(1,\delta) \| <\varepsilon.\]
Thus, $s - (C+1) s_\sB(1,\delta)\in T_3$. Then $\gamma_\delta := s - [0,1]\cdot (C+1) s_\sB(1,\delta) \subset T$ and $s\in\gamma_\delta\cap T_1$ and $s - (C+1) s_\sB(1,\delta)\in\gamma_\delta\cap T_3$, i.e.,
\[s'_\delta = s - c_s(1,\delta) s_\sB(1,\delta)\in T_2\subset \partial^* \cS_\sA.\]
Summarizing, $s = s'_\delta + c_s(1,\delta) s_\sB(1,\delta)$ and $s'_\delta$ has a $k$-atomic representing measure ($k<\cN_\sA$) which has no atom at  $(1,0)$. Therefore, $s$ has an  $l$-atomic presenting measure ($l\leq \cN_\sA$) which has no atom at $(1,0)$. This proves $\cat_\sA(s)\leq\cN_\sA$.
\end{proof}

We illustrate the preceding by the following examples.

\begin{exa}\label{exmp:WORKS}
Let $\sA = \{1,x^2,x^3,x^5,x^6\}$ and  $\sB = \{y^6,x^2 y^4, x^3 y^3, x^5 y, x^6\}$, that is, $m=5$. Then we have
\begin{equation}\label{eq:WORKS-determinant}
\det(s_{\sA}(x),s_{\sA}'(x),s_{\sA}(y),s_{\sA}'(y),s_{\sA}(z))
= (x-y)^4(x-z)^2(y-z)^2 \cdot f(x,y,z),
\end{equation}
where
\begin{equation}
\begin{split}
f(x,y,z):=&\ xy(x^3 y + 4x^2 y^2 + xy^3 + 2x^3 z + 10x^2 y z + 10 x y^2 z\\
&\qquad + 2y^3 z + 4x^2 z^2 + 7 x y z^2 + 4y^2 z^2).
\end{split}
\end{equation}
This implies $\cN_{\sB}=\cN_\sA = 3$ as also  proved in Lemma \ref{lem:onedimNA} and \ref{lemmainhosB}. Hence $\cat_{\sA} \geq 3$. From the Richter--Tchakaloff Theorem (Proposition \ref{richtercor}) we find $\cat_{\sA} \leq m = 5$, while Theorem \ref{worstupperbound} gives a better bound $\cat_{\sA} \leq m - 1 = 4$.

To apply Theorem \ref{thm:onedimCara} we have to check that the assumptions are satisfied. Clearly, $d_1 = 0$ and $d_m=6$ is even. It remains to show that (\ref{condition:one-dim-Caratheodory}) is true. By symmetry it  suffices to verify (\ref{condition:one-dim-Caratheodory}) for
\[f_1(x,y,z) := f(x,y,z),\ f_2(x,y,z) := f(y,z,x),\ \text{and}\ f_3(x,y,z) := f(z,x,y).\]
Set $\cZ := \cZ(f_1)\cap\cZ(f_2)\cap\cZ(f_3)$ and let $X=(x,y,z)\in\cZ$. If $X=0$, then (\ref{condition:one-dim-Caratheodory}) holds. Now let $X\neq 0$. Since $f$ is homogeneous, we can scale $X$ such that $x^2 + y^2 + z^2 = 1$. Then we derive (for instance, by using spherical coordinates) 
%
%
\begin{equation}\label{eq:zeroSetWORKSexample}
\cZ(f_1)\cap\cZ(f_2)\cap\cZ(f_3) = \{(\pm 1,0,0),(0,\pm 1,0),(0,0,\pm 1)\},
\end{equation}
so (\ref{condition:one-dim-Caratheodory}) is fulfilled. Therefore, by Theorem \ref{thm:onedimCara} we have $\cat_\sA = 3$.
\end{exa}

A nice application of the preceding example is the following corollary.

\begin{cor}
Let $p(x) = a + b x^2 + c x^3 + d x^5 + e x^6$ be a non-negative polynomial which is not the zero polynomial. Then $p$ has at most $2$ distinct  real zeros.
\end{cor}
\begin{proof}
Assume to the contrary that $p$ has   three distinct zeros, say $x,y,z$. Let $s$ be the moment sequence of the measure $\mu = \delta_{x} + \delta_{y} + \delta_{z}$. Then $L_s(p)=0$, so $s$ is a boundary point of the moment cone. But from  (\ref{eq:zeroSetWORKSexample}) it follows that the determinant (\ref{eq:WORKS-determinant}) is non-zero, so $s$ is an inner point, a contradiction.
\end{proof}

In the following example the assumption (\ref{condition:one-dim-Caratheodory}) of Theorem \ref{thm:onedimCara} is not satisfied and the assertion (\ref{qbnb}) does not hold.

\begin{exa}\label{exmp:OneDimThmDoesNotWork}
Let $\sA=\{1,x,x^2,x^6\}$ and $\sB=\{y^6,xy^5,x^2y^4,x^6\}$. From Theorem \ref{worstupperbound},  $\cat_{\sB}\leq m-1=3$, while Theorem \ref{densethm} and 
\[\det(s_{\sA}(x),s_{\sA}'(x),s_{\sA}(y),s_{\sA}'(y)) = 2 (y-x)^4 (x+y) (2x^2 + xy + 2y^2)\]
yield $2=\cN_{\sB} \leq \cat_A$, so that $\cat_{\sB} \in\{2,3\}$. We prove that $\cat_{\sB}=3$.  
Let $\nu := \frac{1}{4}(\delta_{-2} + \delta_{-1} + \delta_1 + \delta_2)$. Then
\[s = (s_0, s_1,s_2,s_6)^T  = (s_\sA(-2)+s_\sA(-1)+ s_\sA(1) + s_\sA(2))/4 =
(1, 0, 2.5, 32.5)^T.\]
%
By some straightforward computations it can be shown that $s$ has no $k$-atomic representing measure with $k\leq 2$. Therefore, since $\cat_{\sB} \in\{2,3\}$, we have $\cat_{\sB}=3\neq \left\lceil \frac{3}{2}\right\rceil$. 

Note that $\cN_\sB = 2=\left\lceil \frac{3}{2}\right\rceil$. Thus,  the equality  (\ref{qbnb}) fails.
%
%
%
%
%
%
\end{exa}

\section{Carath\'eorody Numbers: Multidimensional Monomial Case}\label{monomulti}

\begin{dfn}
For $n,d\in\nset$ set
\begin{align}\label{eq:AndDef}
\sA_{n,d} &:= \{x^\alpha : \alpha\in\nset_0^n,\ |\alpha|\leq d\},\\
\sB_{n,d} &:= \{x^\alpha : \alpha\in\nset_0^n,\ |\alpha| = d\}.
\end{align}
\end{dfn}

Note that $|\sA_{n,d}| = \left(\begin{smallmatrix}n+d \\ d\end{smallmatrix}\right)$ and $|\sB_{n,d}| = \left(\begin{smallmatrix}n+d-1 \\ d\end{smallmatrix}\right).$

Throughout this section, we assume the following:
For the polynomials $\cA_{n,d}:=\Lin\, {\sA}_{n,d}$  we consider the truncated moment problem on $\cX=\rset^n$, while for the homogeneous polynomials $\cB_{n,d}:={\Lin}\, \sB_{n,d}$ the moment problem is treated on the real projective space $\cX:=\pset(\rset^{n-1})$. Let $S^{n-1}$ be the unit sphere in $\rset^n, n\geq 2,$ and  $S^{n-1}_+$  the set of points $x\in S^{n-1}$ for which the first non-vanishing coordinate is positive. We consider $S^{n-1}_+$ as a realization of the projective space $\pset(\rset^{n-1})$.

The following simple fact is often used without mention: A polynomial of $\cB_{n,2d}$ is non-negative on $S^{n-1}_+$,  equivalently on $\pset(\rset^{n-1})$, if and only if it is on $\rset^{n-1}$.

The following example shows how differential geometric methods can be used for the truncated moment problem.

\begin{exa}\label{exmp:simpleexampleforNAanduniqueness}
Let $n=d=k=2$, $x^\alpha = (x^{(1)})^{\alpha_1} (x^{(2)})^{\alpha_2}$ and
\begin{align*}
\sA_{2,2} &= \{x^\alpha : \alpha\in\nset_0^2,\ |\alpha|\leq 2\}\\
&= \{x^\alpha : \alpha = (0,0),(0,1),(0,1),(2,0),(1,1),(0,2)\}.
\end{align*}
Then
\[DS_{2,\sA}(C,X) = \begin{pmatrix}
1  & 0 & 0 & 1 & 0 & 0\\
x_1^{(1)} & c_1 & 0 & x_2^{(1)} & c_2 & 0\\
x_1^{(2)} & 0 & c_1 & x_2^{(2)} & 0 & c_2\\
(x_1^{(1)})^2 & 2 c_1 x_1^{(1)} & 0 & (x_2^{(1)})^2 & 2 c_2 x_2^{(1)} & 0\\
x_1^{(1)} x_1^{(2)} & c_1 x_1^{(2)} & c_1 x_1^{(1)} & x_2^{(1)} x_2^{(2)} & c_2 x_2^{(2)} & c_2 x_2^{(1)}\\
(x_1^{(2)})^2  & 0 & 2 c_1 x_1^{(2)} &(x_2^{(2)})^2  & 0 & 2 c_2 x_2^{(2)}
\end{pmatrix},\]
where $C=(c_1,c_2)$ and $X=(x_1,x_2)$,  $x_i = (x_i^{(1)},x_i^{(2)})$. From this we find that
\[\ker DS_{2,\sA}(C,X) = \rset\cdot 
v(C,X) \qquad\text{with}\qquad v(C,X):=\begin{pmatrix}
-2\\
c_1^{-1}(x_1^{(1)}-x_2^{(1)})\\
c_1^{-1}(x_1^{(2)}-x_2^{(2)})\\
2\\
c_2^{-1}(x_1^{(1)}-x_2^{(1)})\\
c_2^{-1}(x_1^{(2)}-x_2^{(2)})
\end{pmatrix}.\]
Hence  $\rank DS_{2,\sA_{2,2}} = 5$ at each point $(x_1,x_2),$  $x_1\neq x_2,$ so the local rank theorem of differential geometry applies. Fix $(C,X)$ as above.  The local rank theorem \cite[Proposition 1, p. 309]{hildeb03} implies  that there is a one-parameter family $(C(t),X(t))$ which has the same moments as $(C,X)$ satisfying the differential equations $\dot{\gamma}(t) = v(C(t),X(t))$ with initial condition $(C(0),X(0))=(C,X)$. This system is
\begin{align*}
\dot{c}_1 &= -2 & \dot{c}_2 &= 2\\
c_1\cdot\dot{x}_1^{(1)}&=x_1^{(1)}-x_2^{(1)} & c_2\cdot\dot{x}_2^{(1)}&=x_1^{(1)}-x_2^{(1)}\\
c_1\cdot\dot{x}_1^{(2)}&=x_1^{(2)}-x_2^{(2)} & c_2\cdot\dot{x}_2^{(2)}&=x_1^{(2)}-x_2^{(2)}
\end{align*}
and its  solution is given by
\begin{align*}
c_1(t) &= c_{1,0} - 2t & x_1^{(i)}(t) &= \gamma_{1,i} + \frac{\gamma_{2,i}}{c_{1,0}+c_{2,0}}\cdot \sqrt{\frac{c_{2,0}+2t}{c_{1,0}-2t}}\\
c_2(t) &= c_{2,0} + 2t &
x_2^{(i)}(t) &= \gamma_{2,i} - \frac{\gamma_{2,i}}{c_{1,0}+c_{2,0}}\cdot \sqrt{\frac{c_{1,0}-2t}{c_{2,0}+2t}}
\end{align*}
with $t\in (-\frac{c_{2,0}}{2},\frac{c_{1,0}}{2}).$ Here $C=(c_{1,0},c_{2,0})$ and $X=((\gamma_{1,1},\gamma_{1,2}),(\gamma_{2,1},\gamma_{2,2}))$ are the initial values at $t=0$. It should be noted that the corresponding moment sequence is indeterminate, but it is a boundary point of the moment cone.
%
\end{exa}

Recall that $\cN_\sA \leq \cat_\sA$ by Theorem \ref{densethm}. There are various other lower bounds for  Carath\'eodory numbers in the literature, see e.g.\ \cite[p.\ 366]{davis84}. In the case $\sA_{2,2k-1}$,  M.\ M\"oller \cite{moller76} obtained the lower bound
\begin{align*}
\mathrm{M\ddot{o}}(2,2k-1) &:= \left(\begin{matrix} k+1\\ 2\end{matrix}\right) + \left\lfloor\frac{k}{2}\right\rfloor.
\end{align*} 
The following result improves  M\"oller's  lower bound.

\begin{prop} 
\begin{equation}\label{eq:NAcomparisson1}
 \mathrm{M\ddot{o}}(2,2k-1) \leq \left\lceil \frac{|\sA_{2,2k-1}|}{3} \right\rceil \leq \cN_{\sA_{2,2k-1}} \leq \cat_{\sA_{2,2k-1}} \quad \text{for}~~k\in \nset.
\end{equation}
For $k\geq 4$ we have
\begin{equation}\label{eq:NAcomparisson2}
\left\lceil \frac{|\sA_{2,2k-1}|}{3} \right\rceil - \mathrm{M\ddot{o}}(2,2k-1) \geq \frac{(k-2)^2-4}{6}.
\end{equation}
\end{prop}
\begin{proof}
The second inequality  of (\ref{eq:NAcomparisson1})   has been stated in Proposition \ref{NAlowerbound}. It reamins to  prove the first inequality of (\ref{eq:NAcomparisson1}). In the cases $k=1,2,3$ it is verified by  direct computations; we omit the details. For $k\geq 4$ it follows from the following computation:
\begin{align*}
\left\lceil \frac{1}{3}  \begin{pmatrix}2k+1\\ 2\end{pmatrix} \right\rceil  &- \left(\begin{pmatrix} k+1\\ 2\end{pmatrix} + \left\lfloor\frac{2}{k}\right\rfloor\right)
\geq \frac{1}{3} \begin{pmatrix}2k+1\\ 2\end{pmatrix} - \left(\begin{pmatrix} k+1\\ 2\end{pmatrix} + \frac{k}{2}\right)\\
&= \frac{(2k+1)k}{3} - \frac{(k+1)k}{2} - \frac{k}{2}
= \frac{(k-2)^2 - 4}{6}.\qedhere
\end{align*}
\end{proof}

Before we turn to our next result we restate the Alexander--Hirschowitz Theorem \cite{alexa95}. We denote by  $V_{n,d,r}$  the vector space of polynomials in $n$ variables of degree at most $d$  having singularities at $r$ general points in $\rset^n$.

\begin{prop}
The subspace $V_{n,d,r}$ has the expected codimension
\[\min \left( r(n+1), \begin{pmatrix} n+d\\ d\end{pmatrix} \right)\]
except for the following cases:
\begin{itemize}
\item[\it (i)] $d=2$; $2\leq r\leq n$, $\codim V_{n,2,r} = r(n+1)- r(r-1)/2$;
\item[\it (ii)] $d=3$; $n=4$, $r=7$, $\dim V_{4,3,7}=1$;
\item[\it (iii)] $d=4$; $(n,r)=(2,5),(3,9),(4,14)$, $\dim V_{n,4,r}=1$.
\end{itemize}
\end{prop}

\begin{thm}\label{numbernaexplicit}
We have
\begin{equation}\label{eq:NAforAllHomPoly}
\cN_{\sA_{n,d}} = \left\lceil \frac{1}{n+1}\begin{pmatrix}n+d\\ n\end{pmatrix} \right\rceil,
\end{equation}
except for the following cases
\begin{itemize}
\item[\em (i)] $d=2$: $\cN_{\sA_{n,2}} = n+1$.
\item[\em (ii)] $n=4$, $d=3$: $\cN_{\sA_{4,3}} = 8$.
\item[\em (iii)] $n=2$, $d=4$: $\cN_{\sA_{2,4}} = 6$.
\item[\em (iv)] $n=3$, $d=4$: $\cN_{\sA_{3,4}} = 10$.
\item[\em (v)] $n=d=4$: $\cN_{\sA_{4,4}} = 15$.
\end{itemize}
\end{thm}
\begin{proof}
From the corresponding definitions of  $V_{n,d,r}$ and $ DS_{k,\sA_{n,d}}$ we obtain
\[\codim V_{n,d,r} = |A_{n,d}| - \dim V_{n,d,r} = \rank DS_{k,\sA_{n,d}}((1,...,1),X).\]
Therefore, apart from exceptional cases, (\ref{eq:NAforAllHomPoly})  follows at once from the Alexander--Hirschowitz Theorem. Next we treat the exceptions.

(i): Note that $\cN_{\sA_{n,2}}\geq \lceil n/2\rceil +1$. Since for all $k$ satisfying  $\lceil n/2\rceil +1 \leq k \leq n$ the matrix $DS_{k,\sA_{n,2}}((1,...,1),X)$ has not the expected full rank for any $X$, the first $k$ with full rank is $k=n+1$.

(ii): Since $\cN_{\sA_{4,3}}\geq 7$ and $DS_{7,\sA_{4,3}}((1,...,1),X)$ has not the expected full rank for any $X$, $\cN_{\sA_{4,3}}= 8$.

(iii): We have $\cN_{\sA_{2,4}}\geq 5$ and $DS_{5,\sA_{2,4}}((1,...,1),X)$ has not the expected full rank for any $X$. Hence  $\cN_{\sA_{2,4}}= 6$.

(iv): Then $\cN_{\sA_{3,4}}\geq 9$ and $DS_{9,\sA_{3,4}}((1,...,1),X)$ has not the expected full rank for any $X$. Thus $\cN_{\sA_{3,4}}= 10$.

(v): Then $\cN_{\sA_{4,4}}\geq 14$ and  $DS_{14,\sA_{4,4}}((1,...,1),X)$ has not the expected full rank for any $X$. Therefore, $\cN_{\sA_{4,4}}= 15$.
\end{proof}

For the homogeneous case we have

\begin{cor}
$\cN_{\sB_{n+1,d}} = \cN_{\sA_{n,d}}$.
\end{cor}
\begin{proof}
Let $X=(X_1,...,X_k)\in\rset^{nk}$, $k=\cN_{\sA_{n,d}}$, be such that $DS_{k,\sA_{n,d}}(1,X)$ has full rank. Then $DS_{k,\sB_{n+1,d}}(1,Y)$ with $Y=((X_1,1),...,(X_k,1))$ has full rank, so that $\cN_{\sA_{n,d}} = k \geq \cN_{\sB_{n+1,d}}$.

On the other hand, let $Y=(Y_1,...,Y_k)\in\rset^{(n+1)k}$, $k=\cN_{\sB_{n+1,d}}$, be such that $DS_{k,\sB_{n+1,d}}(1,Y)$ has full rank. We can assume that all $(n+1)$-th coordinates of $Y_i$ are non-zero by the continuity of the determinant and therefore they can be chosen to be 1, since we are in $\pset(\rset^{n})$. 
The column $\partial_{n+1}s_{\sB_{n+1,d}}(Y_i)$  depends linearly on $s_{\sB_{n+1,d}}(Y_i)$ and $\partial_{j}s_{\sB_{n+1,d}}(Y_i)$, $j=1,...,n$. Therefore, omitting this column does not change the rank.  Hence  $DS_{k,\sA_{n,d}}(1,X)$ with $Y_i=(X_i,1)$ has full rank, that is,  $\cN_{\sA_{n,d}} \leq k=\cN_{\sB_{n+1,d}}$.
\end{proof}


\section{Carath\'eodory Numbers and Zeros of positive Polynomials}\label{carzeropos}

For $f\in {\cB}_{3,2d}$, $\cZ_\pset(f)$ denotes the projective zero set of $f$. Set
\begin{equation}
\alpha(2d) := \frac{3}{2}d(d-1) + 1.
\end{equation}

In this section we use the following proposition of  Choi, Lam, and Reznick \cite{choi80}.

\begin{prop}\label{factoripolpos}
Let $f\in {\cB}_{3,2d}$. Suppose that $f\in {\Pos}(\rset^3)$ and  $|\cZ_{\pset}(f)|>\alpha(2d).$ Then   $|\cZ_{\pset}(f)|$ is infinite and there are polynomials  $p\in {\cB}_{3,2d_1}$, $q\in {\cB}_{3,d_2}$ such that~ $f=pq^2 $,  where $d_1+d_2=d$, $p\in {\Pos}(\rset^3)$, $|\cZ_{\pset}(p)|<\infty $, $q$ is indefinite, and $|\cZ_{\pset}(q)|$ is infinite.  (It is possible that $p$  is a positive real constant; in this case $d_1=0$  and we set ${\cB}_{3,0}:=\rset$.)
\end{prop}

The main aim of this section is to derive  {\it upper bounds} for  the Carath\'eodory number  $\cat_{\sB_{n,2d}}$,  $n=3$. The first approach (Theorem \ref{thm:caraBounds}) applies also to cases with $n>3$ (see Theorem \ref{thm:cara44}). The  second approach  (Theorem \ref{catzeroschoi}) is based on Bezout's Theorem and gives better bounds.

For $d\in \nset$ let $\beta(2d)$ denote the maximum of $|\cZ_\pset(f)|$, where $f\in {\cB}_{3,2d}$,  $f\in {\Pos}(\rset^2)$ and $\cZ_\pset(f)$ is finite. By the Choi--Lam--Reznick Theorem (Proposition \ref{factoripolpos}), $\beta(d)\leq \alpha(d)$ for  $d\in \nset$. We abbreviate $\cat_{2d} := \cat_{\sB_{3,2d}}$.

\begin{thm}
\begin{align}\label{theorem47}
\cat_{2d} \leq\max_{k=0,...,d} \left\{ \begin{pmatrix} 2d+2\\ 2\end{pmatrix} - \begin{pmatrix} 2d+2-k\\ 2\end{pmatrix} + \beta(2(d-k))\right\} + 1.
\end{align}
\end{thm}
\begin{proof}
Let $s\in\cS$. Since the projective space $\pset(\rset^{n-1})$ is compact and condition (\ref{cond+}) holds with $e:= x_1^{2d}+x_2^{2d}+x_3^{2d}$, it follows from Proposition \ref{propstoboundary} that  $\cat_{2d} \leq \max_{s\in\partial\cS} \cat_{2d}(s)+1$. Therefore, it is sufficient to show
\[\cat_{2d} (s) \leq\max_{k=0,...,d} \left\{ \begin{pmatrix} 2d+2\\ 2\end{pmatrix} - \begin{pmatrix} 2d+2-k\\ 2\end{pmatrix} + \beta(2(d-k))\right\} \tag{$*$}\]
for all $s\in\partial\cS$. 

Let $s = \sum_{i=1}^l c_i s_{\sB_{3,2d}}(x_i)$ be an $l$-atomic  representating measure of $s\in\partial\cS.$ Since $s\in\partial\cS,$ there exists a polynomial $p\in \cB_{3,2d}, p\neq 0$, such that $ p(x)\geq 0$ on $\pset(\rset^2)$ and $L_s(p)=0$. Then\ $\supp\mu \subseteq\cZ(p)$, that is, $x_1,\dots,x_l\in \cZ(p)$. 

We can assume without loss of generality that the set $\{s_{\sB_{3,2d}}(x_i)\}_{i=1,...,l}$ is linearly independent. Indeed, assume that these vectors are linearly dependent and let  $\sum_{i=1}^l d_i s_{\sB_{3,2d}}(x_i)=0$ be a non-trivial linear combination. Since all $c_i>0$,  there exists $\varepsilon >0$ such that $c_i + \varepsilon d_i \geq 0$ for all $i$ and $c_j + \varepsilon d_j=0$ for one $j$. Hence  $\mu' = \sum_{i=1}^l (c_i + \varepsilon d_i)\cdot s_{\sB_{3,2d}}(x_i)$ is a  $(l-1)$-atomic representing measure of $s$. 

The polynomial $p\in \cB_{3,2d}$ is non-negative  on $\pset(\rset^2)$, hence on $\rset^3$, so the Choi--Lam--Reznick Theorem (Proposition \ref{factoripolpos}) applies. There are two cases:
\begin{itemize}
\item[a)] $|\cZ(p)|\leq \beta(2d)$.
\item[b)] $p=h^2 q$, where $k:=\deg (h)\geq 1$.
\end{itemize}
In the case a) we have $|\cZ(p)|\leq \beta(2d)$ by the definition of $\beta(2d)$ and therefore $\cat_{2d}(s) \leq \beta(2d)$. This is the case $d=k$ in $(*)$.

Now we turn to case b). Then $k=1,\dots,d-1.$
Let $D(k)$ denote the largest $l$ for which there exist $y_1,\dots,y_l\in \cZ(h)$ such that the vector $s_{\sB_{3,2d}}(y_1),....,s_{\sB_{3,2d}}(y_l)$ are linearly independent. Then, by the  paragraph before last, we have
\begin{align}\label{catdkest}
\cat_{2d}(s) \leq D(k) + \beta(2(d-k)).
\end{align}
Let $y_1,...,y_l\in\cZ(h)$. We   define
\[M(y_1,...,y_l) := \begin{pmatrix}s_{\sB_{3,2d}}(y_1)^T\\ \vdots\\ s_{\sB_{3,2d}}(y_l)^T\end{pmatrix}\]
and $h_\alpha := x^\alpha h$ for  $\alpha\in \nset_0^3, |\alpha|= 2d-k$. Let $\tilde{h}_\alpha$ be the coefficient vector of $h_\alpha$, that is, $h_\alpha(\,\cdot\,) = \langle \tilde{h}_\alpha,s_{\sB_{3,2d}}(\,\cdot\,)\rangle$. Since $s_{\sB_{3,2d}}(y_i)^T\cdot \tilde{h}_\alpha = \langle \tilde{h}_\alpha, s_{\sB_{3,2d}}(y_i)\rangle = y_i^\alpha h(y_i) = 0$, we have  $\tilde{h}_\alpha\in\ker M(y_1,...,y_l)$. Clearly, the vectors  $\tilde{h}_\alpha$ are linearly independent.
Therefore, using (\ref{catdkest}) we derive
\begin{align*}
\cat_{2d}(s)&\leq D(k) + \beta(2(d-k))\\
&\leq \max\rank M(y_1,...,y_l) + \beta(2(d-k))\leq  |\sB_{3,2d}| - |\sB_{3,2d-k}| + \beta(2(d-k))\\
&= \begin{pmatrix} 2d+2\\ 2\end{pmatrix} - \begin{pmatrix} 2d+2-k\\ 2\end{pmatrix} + \beta(2(d-k))
\end{align*}
which is the $k$-th term in ($*$).  

Summarizing, we have $k=d$ in case a) and  $k=1,...,d-1$ in case b). Thus we have proved ($*$) for arbitrary $s\in\partial\cS$ which completes the proof.
\end{proof}

As far as the authors know, the numbers $\beta(2d)$ are not yet known for $d\geq 4$, but we have  $\beta(2d)\leq \alpha(2d)$ by Proposition \ref{factoripolpos}.

\begin{thm}\label{thm:caraBounds} 
For $d\in \nset $ we have
\begin{align}\label{cadalpgada}
\cat_{2d} \leq \alpha(2(d+1)) = \frac{3}{2}d(d+1) + 1.
\end{align}
\end{thm}
\begin{proof}
Since $\beta(2d)\leq \alpha(2d) = \frac{3}{2}d(d-1)+1$ and $(d-k)(k+3)-1\geq 0$ for all $d\in\nset$ and $k=0,...,d-1$, we have for (\ref{theorem47})
\begin{align*}
\frac{3}{2}d(d+1) &= \begin{pmatrix} 2d+2\\2\end{pmatrix} - \begin{pmatrix} d+2\\2\end{pmatrix}\\
&= \begin{pmatrix} 2d+2\\2\end{pmatrix} - \begin{pmatrix} 2d-k+2\\2\end{pmatrix} + \alpha(d-k) + (d-k)(k+3)-1\\
&\geq \begin{pmatrix} 2d+2\\2\end{pmatrix} - \begin{pmatrix} 2d-k+2\\2\end{pmatrix} + \alpha(d-k).
\end{align*}
Inserting the latter into  (\ref{theorem47}) we obtain the assertion.
\end{proof}

In Table \ref{tab:catBounds} we collect some numerical cases of Carath\'eodory bounds.

\begin{table}[!ht]
\begin{tabular}{r|r|rrrr|c}
 & \multicolumn{1}{|c|}{Lower} & \multicolumn{4}{c|}{Upper Bounds for $\cat_{2d}$ from} & known\\
\multicolumn{1}{c|}{$2d$}  & Bounds $\cN_{\sB_{3,2d}}$ & Prop.\ \ref{richtercor} & Thm.\ \ref{worstupperbound} & Thm.\ \ref{thm:caraBounds} &  Thm.\ \ref{catzeroschoi} & $\cat_{2d}$\\\hline
2    &      3 & 6      & 5      &      4 &      4 &  3 \cite{reznick92}\\
4    &      6 & 15     & 14     &     10 &      8 &  6 \cite{reznick92}\\
6    &     10 & 28     & 27     &     19 &     14 & 11 \cite{kunertPhD14}\\
8    &     15 & 45     & 44     &     31 &     22 & --\\
10   &     22 & 66     & 65     &     46 &     32 & --\\
12   &     31 & 91     & 90     &     64 &     47 & --\\
14   &     40 & 120    & 119    &     85 &     65 & --\\
16   &     51 & 153    & 152    &    109 &     86 & --\\
18   &     64 & 190    & 189    &    136 &    110 & --\\
20   &     77 & 231    & 230    &    166 &    137 & --\\
40   &    287 & 861    & 860    &    631 &    572 & --\\
100  &   1717 & 5151   & 5150   &   3826 &   3677 & --\\
1000 & 167167 & 501501 & 501500 & 375751 & 374252 & --\\
\end{tabular}
\caption{Bounds on the Carath\'eodory numbers $\cat_{2d}$ for $d=1,...,10,20,50,500$ from Proposition \ref{richtercor} and Theorems \ref{worstupperbound}, \ref{thm:caraBounds},  \ref{catzeroschoi}.\label{tab:catBounds}}
\end{table}

The next proposition is also due to Choi--Lam--Reznick \cite{choi80}. We will use it to derive a  bound for the Carath\'eodory number $\cat_{\sB_{4,4}}$.

\begin{prop}\label{prop:ChoiLamReznick2}
If $p\in \cB_{4,4}$  and $|\cZ_\pset(p)|>11$, then $p$ is a sum of at most six squares of quadratics.
\end{prop}

\begin{thm}\label{thm:cara44}
$\cat_{\sB_{4,4}}\leq 26$.
\end{thm}
\begin{proof}
Let $s$ be a boundary moment sequence. Then there exists $p\in  \cB_{4,4}, p\neq 0,$ such that $p\in {\Pos}(\rset^3)$ and  $L_s(p)=0$. By Proposition \ref{prop:ChoiLamReznick2},   $|\cZ(p)|\leq 11$ or we have $p=f_1^2+...+f_6^2$ for some $f_1,...,f_6\in \cB_{4,2}$.  
In the following proof we give an upper bound on the maximal number $l$ of linearly independent vectors $s_{\sB_{4,4}}(x_1),...,s_{\sB_{4,4}}(x_l)$ with $x_i\in\cZ(p).$ By Theorem \ref{thm:LinIndepZeros}, this number $l$ is an upper bound of $\cat_{\sB_{4,4}}(s)$. We proceed in a similar manner  as in the proof of Theorem \ref{thm:caraBounds}.

By Proposition \ref{prop:ChoiLamReznick2} we have two cases:
\begin{itemize}
\item[a)] $|\cZ(p)|\leq 11$, 
\item[b)] $p= f_1^2 + ... + f_k^2$, $k\leq 6$.
\end{itemize}

In the case a) we clearly have  $l\leq |\cZ(p)|\leq 11$. 

Now we treat
case b). Clearly, $\cZ(f_1^2+...+f_k^2)\subseteq\cZ(f_1^2)=\cZ(f_1).$
Hence it suffices to determine the maximal number  $l$ for a single square $p=f^2$, where $f\in\sB_{4,2}, f\neq 0$. Let $x_1,...,x_l\in\cZ(f)$ be such that the set $\{s_{\sB_{4,4}}(x_i)\}_{i=1,...,l}$ is linearly independent. Define
\[M(x_1,...,x_l) := \begin{pmatrix}s_{\sB_{4,4}}(x_1)^T\\ \vdots\\ s_{\sB_{4,4}}(x_l)^T\end{pmatrix},\]
$f_\alpha := x^\alpha f$ for  $\alpha\in \nset_0^4, |\alpha| = 2$, and $\tilde{f}_\alpha$ by $f_\alpha(\,\cdot\,) = \langle \tilde{f}_\alpha,s_{\sB_{4,4}}(\,\cdot\,)\rangle$. Then we have $\tilde{f}_\alpha\in\ker M(x_1,...,x_l)$, since $s_{\sB_{4,4}}(x_i)^T\cdot \tilde{f}_\alpha = \langle \tilde{f}_\alpha, s_{\sB_{4,4}}(x_i)\rangle = x_i^\alpha f(x_i) = 0$. The vectors $\tilde{f}_\alpha$ are linearly independent. Therefore, $\dim\ker M(x_1,...,x_l) \geq \# f_\alpha = |\sB_{4,2}|$ and
\[l = \rank M(x_1,...,x_l)\leq |\sB_{4,4}| - |\sB_{4,2}| = \begin{pmatrix}7\\ 3\end{pmatrix} - \begin{pmatrix}5\\ 3\end{pmatrix} = 25.\]
This proves that the moment sequence $s$ can be represented by at most 25 atoms.

Summarizing both cases, we have shown that each $s\in\partial\cS$ has a  $k$-atomic representing measure with $k\leq 25$. Therefore, by Proposition \ref{propstoboundary}, $\cat_{\sB_{4,4}}\leq 25 + 1 = 26$.
\end{proof}

Proposition \ref{richtercor} yields $\cat_{\sB_{4,4}}\leq 35$, while Theorem \ref{worstupperbound} gives $\cat_{\sB_{4,4}}\leq 34$. Combining the upper bound of Theorem \ref{thm:cara44} with the lower bound from Theorem \ref{densethm} we get
\begin{equation}
\cN_{\sB_{4,4}} = \cN_{\sA_{3,4}} = 10 \leq \cat_{\sB_{4,4}} \leq 26.
\end{equation}

Now we give another approach to obtain estimates of the Carath\'eodory number $\cat_{\sB_{3,2d}}$ from above. It is based on Bezout's Theorem.

Let $f_1\in {\cB}_{3,d_1}$ and $f_2\in {\cB}_{3,d_2}$. 
For each point $t\in \cZ_{\pset}(f_1)\cap \cZ_{\pset}(f_2)$ the {\it intersection multiplicity} $I_t(f_1,f_2)\in \nset$ of the projective curves $f_1=0$ and $f_2=0$ at $t$ is defined in \cite[III,  Section 2.2]{walkerAlgCurves}.
We do not restate the precise definition here. In what follows we use only the fact that $I_t(f_1,f_2)\geq 2$ if $t$ is a singular point of one of the curves $f_1=0$ or $f_2=0$.

We use the following version of \textit{Bezout's Theorem}. The symbol $|Z|$ denotes the  number of points of a set $Z$.

\begin{lem}\label{bezoutstheorem}
If $f_1\in {\cB}_{3,d_1}$ and $f_2\in {\cB}_{3,d_2}$ are relatively prime in $\rset[x_1,x_2,x_3]$, then 
\begin{align*}
\sum\nolimits_{t\in \cZ_{\pset}(f_1)\cap \cZ_{\pset}(f_2)}~ I_t(f_1,f_2)\leq d_1 d_2.
\end{align*}
\end{lem}
\begin{proof}
See e.g. \cite[p. 59]{walkerAlgCurves}.
\end{proof}

\begin{lem}\label{zerolemmabezout}
Let $s$ be a moment sequence for $\sB_{3,2d}$. Suppose $p\in  \cB_{3,k}$ is irreducible in $\rset[x_1,x_2,x_3]$, $k\leq d$, and  $L_s(p^2 (x_1^2+x_2^2+x_3^2)^{d-k})=0$. Then
\[\cat_{2d}(s)\leq d k + 1.\]
\end{lem}
\begin{proof}
Consider the moment cone $\tilde{\cS}:=\cS(\sB_{3,2d},\cZ(p))$.  Then $\tilde{\cS}$   is an exposed face of the moment cone $\cS=\cS(\sB_{3,2d},\pset(\rset^2))$ and $s\in \tilde{\cS}$. By  Proposition \ref{limitcar},  $\cS$ is closed and  so is $\tilde{\cS}$.  Clearly, each point of $\tilde{\cS}$ is the limit of relative inner points of $\tilde{\cS}$. Therefore, since the sets $\tilde{\cS}_k$ are closed  by  Proposition \ref{limitcar},  it is sufficient to prove the assertion for all relatively inner points of the cone $\tilde{\cS}$.

Let $s$ be a relatively  inner point of $\tilde{\cS}$ and $x\in\cZ(p)$. Setting $e:=x_1^{2d} + x_2^{2d} + x_3^{2d}$,  condition (\ref{cond+}) holds. Since $\cZ(p)$ is compact,  Proposition \ref{propstoboundary} applies, so the supremum $c_s(x) := \sup \, \{c: s-c\cdot s_{\sB_{3,2d}}(x)\in\tilde{\cS}\}$ is attained and $s' := s - c_s(x)\cdot s_{\sB_{3,2d}}(x) \in \partial \tilde{\cS}$.
Thus there exists a supporting hyperplane of the cone $\tilde{\cS}$ at $s'$. Hence there exists a polynomial $q\in\cB_{3,2d}$ such that $L_{s'}(q)=0$, $L_s(q)>0$, and $q\geq 0$ on $\cZ(p)$. 
From $L_{s'}(q)=L_s(q)-c_s(x)q(x)=0$ it follows that $q(x)\neq 0$. (Indeed,  otherwise $L_s(q)=0$, so  $s$ would be a boundary point of  $\tilde{\cS},$ a contradiction.) Since $p(x)=0$ and $q(x)\neq 0$, the irreducible polynomial $p$ is not a factor of $q$, so $p$ and $q$ are relatively prime and Bezout's Theorem applies.

Since $q(y)\geq 0$ on $\cZ(p)$, for each intersection point of $q$ and $p$ has the intersection multiplicity of at least $2$. Therefore,  by Lemma \ref{bezoutstheorem},
\begin{align}\label{zpzq}
2|\cZ(q)\cap \cZ(p)|\leq \deg (q) \deg (p)= 2d k.
\end{align}
Since each representing measure of $s'$ is supported on $\cZ(p)\cap \cZ(q)$, (\ref{zpzq}) implies that $\cat_{2d}(s')\leq  d k$. Hence $\cat_{2d}(s)\leq \cat_{2d}(s') + 1 \leq d k + 1$.
\end{proof}

Our main result in  this section is the following theorem.

\begin{thm}\label{catzeroschoi}
$\cat_{2d} \leq \alpha(2d) + 1 = \frac{3}{2}d(d-1) + 2$ for $d\in \nset$, $d\geq 5$.
\end{thm}
\begin{proof}
Let us consider the moment cone $\cS:=\cS(\sB_{3,2d},\pset(\rset^2))$. We proceed in a similar manner as in the proof of 
Lemma \ref{zerolemmabezout}. By   Proposition \ref{limitcar}, the sets $\cS_k$ are closed.  Hence it suffices  to prove the inequality  $\cat_{2d}(s) \leq \alpha(2d) + 1 $ for all relatively inner points of the cone $\cS$.

Let $s$ be an inner point of $\cS$ and $x\in\pset(\rset^2)$.   By Proposition \ref{propstoboundary}, the supremum  $c_s(x) := \sup\, \{c:s-c\cdot s_{\sB_{3,2d}}(x)\in\cS\}$ is attained and  $s' := s - c_s(x)\cdot s_{\sB_{3,2d}}(x)\in\partial\cS$. Then there exists a supporting hyperplane of $\cS$ at $s$, hence there is a polynomial  $f\in\cB_{3,2d}$ such that $L_{s'}(f)=0$ and $f\geq 0$ on $\pset(\rset^2)$. We apply  Proposition \ref{factoripolpos} to $f$. Then, we can write $f=p\cdot q_1^2 \cdots q_r^2$ ($r\leq d$), where $p\in {\Pos}(\pset(\rset^2))$, all $q_i$ are indefinite and irreducible in $\rset[x_1,x_2,x_3]$,  $\cZ(p)<\infty $ and all $|\cZ(q_i)| $ are infinite.
Since
\[\cZ(f) = \cZ(p)\cup \cZ(q_1)\cup\cdots\cup\cZ(q_r)\]
we find a disjoint decomposition $Z\cup Z_1\cup\cdots\cup Z_r$ of $\cZ(f)$ with $Z\subseteq\cZ(p)$ and $Z_i\subseteq\cZ(q_i)$. Let $\mu' = \sum_{j=1}^m c_j \delta_{x_j}$ be a representing measure of $s'$ and set%
\[s_0 := \sum_{x_j\in Z} c_j s_{\sB_{3,2d}}(x_j)\quad \text{and}\quad s_i := \sum_{x_j\in Z_i} c_j s_{\sB_{3,2d}}(x_j).\]
Clearly, $s' = s_0 + s_1 + \cdots + s_r$. Setting $d_i = \deg(q_i)$ and $2k = \deg(p)$, we have $d = k + d_1 + \cdots + d_r$ and $r\leq d-k$. Using Proposition \ref{factoripolpos} and Lemma \ref{zerolemmabezout} we derive
\begin{align*}
\cat_{\sB_{3,2d}}(s') &\leq \cat_{\sB_{3,2d}}(s_0) + \cat_{\sB_{3,2d}}(s_1) + \cdots + \cat_{\sB_{3,2d}}(s_r)\\
&\leq \alpha(2k) + (d\cdot d_1 + 1) + \cdots + (d\cdot d_r + 1) = \alpha(2k) + d(d-k) + r\\
&\leq \alpha(2k) + (d+1)(d-k) = \alpha(2d) - \underbrace{(\alpha(2d)-\alpha(2k)-(d+1)(d-k))}_{=\frac{1}{2}(d-k)(d+3k-5)\geq 0\ \forall d\geq 5,\ k=0,...,d}\\
&\leq \alpha(2d).
\end{align*}
Therefore, $\cat_{\sB_{3,2d}}(s) \leq \cat_{\sB_{3,2d}}(s') + 1 \leq \alpha(2d)+1 = \frac{3}{2}d(d-1)+2$ for all $d\geq 5$.
\end{proof}

\begin{exa}[$d=5$]\label{exmp:HarrisPolynomial}
W.\ R.\ Harris \cite{harris99} discovered a  polynomial $h\in \sB_{3,10}$  that is nonnegative on $\pset(\rset^2)$ with projective zero set
\[\cZ_\pset(h) = \{(1,1,0)^*, (1,1,\sqrt{2})^*,(1,1,1/2)^*\},\]
where $(a,b,c)^*$ denotes all permutations of $(a,b,c)$ including sign changes. Hence $h$ has exactly $30$  projective zeros $z_i,i=1,\dots,30.$ A computer calculation shows that the matrix $(s_{\sB_{3,10}}(z))_{z\in\cZ_\pset(h)}$ has rank $30$, i.e., the set $\{ s_{\sB_{3,10}}(z_i):i=1,\dots,30\} $ is linearly independent. Therefore, $\cat_{\sB_{3,10}} \geq 30 $ by Theorem \ref{thm:LinIndepZeros}.  Further, we compute $ \cN_{\sB_{3,10}} = 15$ and have $ \cat_{\sB_{3,10}}\leq \alpha(10)+1=37$ by Theorem \ref{catzeroschoi}. Summarizing, 
\begin{align}\label{casen=5}
\cN_{\sB_{3,10}} = 15 < 30\leq  \cat_{\sB_{3,10}}\leq 37.
\end{align}
\end{exa}

The following corollary reformulates Theorem \ref{thm:LinIndepZeros} in the present context.

\begin{cor}\label{betacarest0}
Let $d\in\nset$ and $p\in\cB_{3,2d}$. Suppose that $p\in {\Pos}(\rset^3)$, $|\cZ(p)|=\beta(2d)$, and the set $\{s_{\sB_{3,2d}}(z) : z\in\cZ(p)\}$ is linearly independent. Then
\[\beta(2d)\leq \cat_{\sB_{3,2d}}.\]
\end{cor}

It seems  natural to ask whether or not the assumption on the linear independence of the set $\{s_{\sB_{3,2d}}(z) : z\in\cZ(p)\}$ in Corollary  \ref{betacarest0} can be omitted. This leads to the

%
\begin{center}\it
{\bf Question:} Suppose   $p\in \cB_{3,2d}$,   $p\in \Pos(\rset^3)$, and $|\cZ(p)|<\infty$ (or $|\cZ(p)|=\beta(2d)$).\\
Is the set $\{s_{\sB_{3,2d}}(z) : z\in\cZ(p)\}$ linearly independent?
\end{center}

Note that for the Robinson polynomial $R\in \cB_{3,6}$ the answer is ``Yes''.

\smallskip

Recall that $\beta(2d)\leq \alpha (2d)$ by the Choi--Lam--Reznick Theorem (Proposition \ref{factoripolpos}). It seems likely to conjecture that
\begin{equation}\label{Conjecture}
{\bf Conjecture:}\quad \beta(2d) \leq \cat_{\sB_{3,2d}}\leq \beta(2d)+1\quad {\rm for} ~~~~ d\geq 3.
\end{equation}
The Robinson polynomial has $10$ projective  zeros, so that $\alpha(6)=\beta(6)=10.$ Therefore, since $ \cat_{\sB_{3,6}}=11$ as shown in \cite{kunertPhD14}, this conjecture is true for $d=3$. 
As noted above, the Harris polynomial $R\in \cB_{3,10}$  has $30$ projective zeros. Hence $30\leq \beta (10)\leq \alpha(10)=31$. 

From the  proof of Theorem \ref{catzeroschoi}  it follows that (\ref{Conjecture}) holds if
\[\beta(d)+ (d'+1)(d'-d) \leq \beta(d')\quad {\rm for}\quad d'\in \nset, d\in \nset_0, d<d', (d',d)\neq (3,0).\]

\section{Carath\'eodory Numbers and Real Waring Rank}\label{realwaringrank}

In Definition \ref{def:signedCara} we introduced the signed Carath\'eodory number $\cat_{\sA,\pm}$. In this section we connect it to the \emph{real Waring rank} $w(n,2d)$, that is, to the smallest number $w(n,2d)$ such that each $f\in\cB_{n,2d}$ can be written as  real linear combination
\begin{equation}\label{eq:apolarScalar0}
f(x) = \sum_{i=1}^k c_i (x\cdot \lambda_i)^{2d}
\end{equation}
of $2d$-powers of linear forms $x\cdot \lambda_i = \lambda_{i,1}x_1+\cdots+\lambda_{i,n}x_n$, where $k\leq w(n,2d)$, $c_i\in\rset$, $\lambda_i\in\rset^n$.

Let us recall some basics on the \emph{apolar scalar product} $[\,\cdot\,,\,\cdot\,]$, see e.g.\ \cite{reznick92}. For $\alpha=(\alpha_1,...,\alpha_n)\in\nset_0^n$ with $|\alpha|:=\alpha_1+\cdots+\alpha_n=2d$ we set $\gamma_\alpha := \frac{(2d)!}{\alpha_1!\cdots\alpha_n!}$. Let $p,q\in \cB_{n,2d}$. We write $p(x) = \sum_\alpha \gamma_\alpha a_\alpha x^\alpha$ and $q(x) = \sum_\alpha \gamma_\alpha b_\alpha x^\alpha$ and define
\[[p,q] := \sum_\alpha \gamma_\alpha a_\alpha b_\alpha.\]
Then $(\cB_{n,2d},[\,\cdot\, ,\,\cdot\,])$ becomes a finite-dimensional real Hilbert space. Setting $f_\lambda(x) := (\lambda\cdot x)^{2d}$, we obtain
\begin{equation}\label{eq:apolarScalar1}
[p,f_\lambda] = \sum_\alpha \gamma_\alpha a_\alpha \lambda^\alpha = p(\lambda).
\end{equation}
Let $f$ be of the form (\ref{eq:apolarScalar0}). Then, for $p\in \cB_{n,2d}$ it follows from (\ref{eq:apolarScalar1}) that
\begin{equation}\label{eq:apolarScalar2}
L_f(p):=[f,p] =\left[\sum\nolimits_i c_i f_{\lambda_i},p\right] =\sum_{i=1}^k c_i p(\lambda_i),
\end{equation}
%
that is, the linear functional $L_f$ on $\cB_{n,2d}$ is the integral with respect to the signed measure $\mu:=\sum_{i=1}^k c_i \delta_{\lambda_i}$. Conversely, each signed atomic measure yields a function $f$ of the form (\ref{eq:apolarScalar0}) such that (\ref{eq:apolarScalar2}) holds. By the Riesz Theorem all linear functionals on $\cB_{n,2d}$ are of the form $L_f$, where $f$ is as in (\ref{eq:apolarScalar0}).

\begin{thm}
\begin{itemize}
\item[\em (i)] $w(n,2d) = \cat_{\sB_{n,2d},\pm}$.

\item[\em (ii)] $\cN_{\sB_{n,2d}} \leq w(n,2d) \leq 2\cN_{\sB_{n,2d}}$.

\item[\em (iii)] Set $N := \cN_{\sB_{n,2d}}$. Then there exists $\lambda=(\lambda_1,...,\lambda_N)\in \rset^{N\cdot n}$ such that for all $\varepsilon>0$ and $p\in\cB_{n,2d}$ we have
\[p(x) = c\cdot \sum_{i=1}^{\cN_{\sB_{n,2d}}} \left[ (\lambda_i\cdot x)^{2d} - c_i (\lambda_i^\varepsilon\cdot x)^{2d}\right]\]
for some $\lambda^\varepsilon = (\lambda_1^\varepsilon,...,\lambda_N^\varepsilon)$ with $\|\lambda-\lambda^\varepsilon\|<\varepsilon$, $|1-c_i|<\varepsilon$, $c\in\rset$.

\item[\em (iv)] The set of vectors $\lambda$ as in  (iii) is open and dense in $\rset^{\cN_{\sB_{n,2d}} \cdot n}$.
\end{itemize}
\end{thm}
\begin{proof}
(i) is clear  from the preceding considerations on the apolar scalar product.
  Remark \ref{rem:signedCarabounds} and (i) imply (ii), while (iii) follows from Theorem \ref{thm:signedCaraUpperBound} combined with (i). (iv) is a consequence of Sard's Theorem as in Theorem \ref{densethm}.
\end{proof}

With Theorem \ref{numbernaexplicit} the upper bound in (ii) was already obtained  in \cite[Cor.\ 9]{blekhe15}.

\section*{Acknowledgment}

The authors are grateful to G.\ Blekherman, M.\ Schweighofer, and C.\ Riener for valuable discussions on the subject of this paper.
K.S.\  thanks also J.\ St{\"u}ckrad to helpful discussions.
P.dD.\ was supported by the Deutsche Forschungsgemeinschaft (SCHM1009/6-1).


\end{document}